\documentclass[twoside,leqno]{article}

\usepackage[letterpaper]{geometry}

\usepackage{siamproceedings}

\usepackage[T1]{fontenc}
\usepackage{graphicx}
\usepackage{epstopdf}
\usepackage{enumitem}
\usepackage{algorithmic}
\usepackage{xcolor}
\usepackage{listings}
\usepackage{braket}
\usepackage{amstext,amsfonts,bm,amssymb}
\usepackage{thm-restate}
\usepackage{mathtools}
\usepackage{qcircuit}
\usepackage{tikz}
\ifpdf
  \DeclareGraphicsExtensions{.eps,.pdf,.png,.jpg}
\else
  \DeclareGraphicsExtensions{.eps}
\fi
\usepackage{pgf,tikz,ifthen,calc, tkz-euclide}
\usepackage{amsopn}

\newsiamremark{remark}{Remark}
\newsiamremark{hypothesis}{Hypothesis}
\crefname{hypothesis}{Hypothesis}{Hypotheses}

\definecolor{mintedbg}{RGB}{248,248,248}
\definecolor{mintedframe}{RGB}{220,220,220}
\definecolor{mintednum}{RGB}{128,128,128}
\definecolor{mintedkeyword}{RGB}{0,0,180}
\definecolor{mintedcomment}{RGB}{0,128,0}
\definecolor{mintedstring}{RGB}{160,32,240}

\lstdefinestyle{mintedstyle}{
  backgroundcolor=\color{mintedbg},
  basicstyle=\ttfamily\small,
  keywordstyle=\color{mintedkeyword}\bfseries,
  commentstyle=\color{mintedcomment}\itshape,
  stringstyle=\color{mintedstring},
  numberstyle=\tiny\color{mintednum},
  numbers=left,
  stepnumber=1,
  numbersep=8pt,
  frame=single,
  framerule=0.4pt,
  rulecolor=\color{mintedframe},
  breaklines=true,
  breakatwhitespace=false,
  columns=fullflexible,
  keepspaces=true,
  showstringspaces=false,
  tabsize=2,
  captionpos=b
}

\usepackage[dvipsnames]{xcolor}

\definecolor{darkGreen}{HTML}{006400}
\definecolor{myPurple}{HTML}{8B008B}
\definecolor{myOrange}{HTML}{FF8C00}

\definecolor{linkblue}{named}{MidnightBlue}
\hypersetup{colorlinks=true, linkcolor=linkblue, anchorcolor=linkblue,
  citecolor=linkblue, filecolor=linkblue, menucolor=linkblue,
  urlcolor=linkblue} 

\newsiamthm{open}{\color{BrickRed}Open problem}
\newsiamthm{question}{Question}

\definecolor{darkblue}{rgb}{0,0,0.7} 
\definecolor{darkred}{rgb}{0.7,0,0} 
\definecolor{darkgreen}{rgb}{0,0.5,0} 

\newcommand{\defn}[1]{\textsl{\color{darkblue} #1}} 
\let\emph\relax 
\DeclareTextFontCommand{\emph}{\color{OliveGreen}\it}
\tikzset{MyNode/.style={circle, draw, inner sep=2,outer sep=0, fill=gray}}

\title{\Large Biclique decompositions from Welzl orders}

 \author{Jean Cardinal\thanks{Universit\'e libre de Bruxelles (ULB), Brussels, Belgium. Supported by the Fonds de la Recherche
Scientifique-FNRS under Grant No.~T003325F. \texttt{jean.cardinal@ulb.be}} 
\and Rose McCarty\thanks{Georgia Institute of Technology, Atlanta, Georgia, USA. Supported by the National Science Foundation under Grant No.~DMS-2452111. \texttt{rmccarty3@gatech.edu}} 
\and Yelena Yuditsky\thanks{University of Leeds, Leeds, UK. \texttt{y.yuditsky@leeds.ac.uk}}}

\date{}

\begin{document}
\maketitle

\begin{abstract}
    A biclique decomposition of a graph is a partition of its edges into complete bipartite subgraphs. We consider graphs whose vertices can be ordered such that the neighborhood of every vertex is the union of a sublinear number of intervals. We observe that these graphs admit compact representations in the form of biclique decompositions of small size. Here, the size of a decomposition is measured as the sum of the number of vertices of its bicliques. Combining this result with the existence of suitable vertex orderings for graphs of low neighborhood complexity, as proven by Welzl in 1988, we recover and extend several known results up to logarithmic factors. These results include upper bounds on the Zarankiewicz problem, matrix multiplication, shortest path algorithms, and quantum circuit complexity  in ``well-structured'' instances.
\end{abstract}

\section{Introduction}

We consider the problem of finding biclique decompositions of graphs, which serve as convenient compact representations. A \defn{biclique} of a graph $G$ is a complete bipartite subgraph of $G$.
A \defn{biclique cover} of $G$ is a collection of bicliques of $G$ such that every edge of $G$ is contained in at least one of those bicliques. A \defn{biclique decomposition} is a biclique cover in which every edge is covered exactly once. 
We define the \defn{size} of a biclique cover $\mathcal{B}$ as the sum $\sum_{B\in\mathcal{B}} |V(B)|$ of the number of vertices of its bicliques. Equivalently, the size of a biclique cover is the sum over all vertices of $G$ of the number of bicliques containing that vertex. See~\Cref{fig:bicDecomp} for an example. 

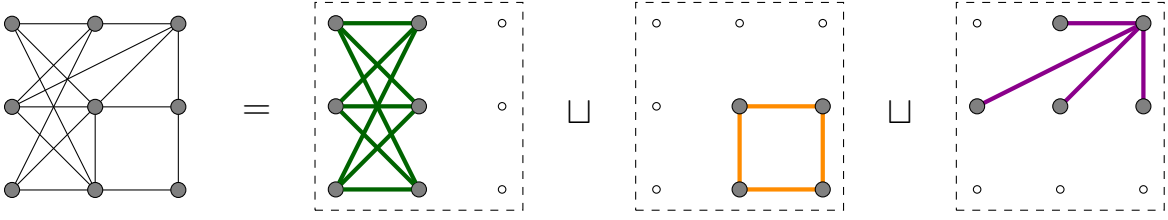
\begin{figure}[t]
\centering
\begin{tikzpicture}[scale=.55, every node/.style={MyNode}]
    \def \width {2cm}
    \foreach \i in {1,2,3}{
        \foreach \j in {1,2,3}{
			\node (a\i\j) at (\i*\width, \j*\width) {};
		}
	}
    \foreach \i in {1,2,3}{
        \foreach \j in {1,2,3}{
			\draw[] (a1\i) -- (a2\j) {};
		}
	}
	\draw[] (a22) -- (a32) {};
    \draw[] (a32) -- (a31) {};
    \draw[] (a31) -- (a21) {};
    \draw[] (a21) -- (a22) {};
	\draw[] (a33) -- (a22) {};
    \draw[] (a33) -- (a12) {};
    \draw[] (a33) -- (a23) {};
    \draw[] (a33) -- (a32) {};
    \draw[draw=none] (-.5cm+\width, -.5cm+\width) --++ (2*\width+1cm, 0) --++ (0, 2*\width+1cm) --++ (-2*\width-1cm, 0) --++ (0, -2*\width-1cm);
\end{tikzpicture}\hskip .5cm
\begin{tikzpicture}[scale=.6, every node/.style={MyNode}]
    \def \width {2cm}
    \node[draw=none, fill=none] at (0, 0) {};
    \node[draw=none, fill=none] at (0, 2*\width) {};
    \node[rectangle, draw=none, fill=none] (center) at (0, \width) {\Large$=$};
\end{tikzpicture}\hskip .5cm
\begin{tikzpicture}[scale=.55, every node/.style={MyNode}]
    \def \width {2cm}
    \foreach \i in {1,2}{
        \foreach \j in {1,2,3}{
			\node (a\i\j) at (\i*\width, \j*\width) {};
		}
	}
        \foreach \j in {1,2,3}{
			\node[fill=white, circle, inner sep=1,outer sep=0] (a3\j) at (3*\width, \j*\width) {};
		}
    \foreach \i in {1,2,3}{
        \foreach \j in {1,2,3}{
			\draw[ultra thick, darkGreen] (a1\i) -- (a2\j) {};
		}
	}
    \draw[dashed] (-.5cm+\width, -.5cm+\width) --++ (2*\width+1cm, 0) --++ (0, 2*\width+1cm) --++ (-2*\width-1cm, 0) --++ (0, -2*\width-1cm);
\end{tikzpicture}\hskip .5cm
\begin{tikzpicture}[scale=.6, every node/.style={MyNode}]
    \def \width {2cm}
    \node[draw=none, fill=none] at (0, 0) {};
    \node[draw=none, fill=none] at (0, 2*\width) {};
    \node[rectangle, draw=none, fill=none] (center) at (0, \width) {\Large$\sqcup$};
\end{tikzpicture}\hskip .5cm
\begin{tikzpicture}[scale=.55, every node/.style={MyNode}]
    \def \width {2cm}
    \foreach \i in {1,2,3}{
        \foreach \j in {1,2,3}{
			\node[fill=white, circle, inner sep=1,outer sep=0] (a\i\j) at (\i*\width, \j*\width) {};
		}
	}
	\draw[ultra thick, myOrange] (a22) -- (a32) {};
    \draw[ultra thick, myOrange] (a32) -- (a31) {};
    \draw[ultra thick, myOrange] (a31) -- (a21) {};
    \draw[ultra thick, myOrange] (a21) -- (a22) {};
    \node at (a21) {};
    \node at (a31) {};
    \node at (a22) {};
    \node at (a32) {};
    \draw[dashed] (-.5cm+\width, -.5cm+\width) --++ (2*\width+1cm, 0) --++ (0, 2*\width+1cm) --++ (-2*\width-1cm, 0) --++ (0, -2*\width-1cm);
\end{tikzpicture}\hskip .5cm
\begin{tikzpicture}[scale=.6, every node/.style={MyNode}]
    \def \width {2cm}
    \node[draw=none, fill=none] at (0, 0) {};
    \node[draw=none, fill=none] at (0, 2*\width) {};
    \node[rectangle, draw=none, fill=none] (center) at (0, \width) {\Large$\sqcup$};
\end{tikzpicture}\hskip .5cm
\begin{tikzpicture}[scale=.55, every node/.style={MyNode}]
    \def \width {2cm}
    \foreach \i in {1,2,3}{
        \foreach \j in {1,2,3}{
			\node[fill=white, circle, inner sep=1,outer sep=0] (a\i\j) at (\i*\width, \j*\width) {};
		}
	}
	\draw[ultra thick, myPurple] (a33) -- (a22) {};
    \draw[ultra thick, myPurple] (a33) -- (a12) {};
    \draw[ultra thick, myPurple] (a33) -- (a23) {};
    \draw[ultra thick, myPurple] (a33) -- (a32) {};
    \node at (a33) {};
    \node at (a12) {};
    \node at (a22) {};
    \node at (a32) {};
    \node at (a23) {};
    \draw[dashed] (-.5cm+\width, -.5cm+\width) --++ (2*\width+1cm, 0) --++ (0, 2*\width+1cm) --++ (-2*\width-1cm, 0) --++ (0, -2*\width-1cm);
\end{tikzpicture}
\caption{A graph with a biclique decomposition into three bicliques of total size $6 + 4 + 5 = 15$. \label{fig:bicDecomp}}
\end{figure}
Our main result is that graphs of polynomial neighborhood complexity have biclique decompositions of small size. The \defn{neighborhood complexity} of a graph $G$ is the function $\eta_G$ of $m$ defined as the maximum number of distinct neighborhoods within a set of size $m$:
\[
\eta_G(m) = \max_{C\subseteq V(G), |C|=m} |\{ N(v)\cap C : v\in V(G) \}|,
\]
where $N(v)$ is the set of neighbors of $v$. This function is also known as the \defn{shatter function} of the neighborhood set system. The notion of neighborhood complexity was inspired by the Sauer-Shelah Lemma~\cite{Sauer1972, Shelah1972}, which says that any set system of VC-dimension $d$ has shatter function $O(m^d)$. In general, graph classes with polynomial neighborhood complexity capture many ``well-structured'' graph classes.

We prove that such ``well-structured'' graph classes admit small biclique decompositions.

\begin{theorem}
\label{thm:main}
There is an $O(n^2\log n)$-time randomized algorithm that computes a biclique decomposition $\mathcal{B}$ of an $n$-vertex graph $G$ so that $\mathcal{B}$ has size:
    \begin{itemize}
        \item $O(n^{2-1/d}\log n)$ if $G$ has neighborhood complexity $\eta_G(m)\leq c m^d$ for fixed $c>0$ and $d>1$, 
        \item $O(n\log^2 n)$ if $G$ has neighborhood complexity $\eta_G(m)\leq c m$ for fixed $c>0$.
    \end{itemize}
\end{theorem}

Theorem~\ref{thm:main} improves on a recent upper bound of $O(n^{2-1/(d+1)})$ due to Krapivin, Przybocki, Sanhueza-Matamala, and Subercaseaux~\cite[Theorem 39]{KPSS25}.
Our proof relies on the existence of \defn{Welzl orders}~\cite{MR1213461}, that is, orderings of the vertices of a graph so that the neighborhood of each vertex is the union of a sublinear number of intervals. We also use a recent refinement of Welzl's theorem for the case of linear shatter functions due to Bonnet, Duron, Sylvester, and Zamaraev~\cite{MR4868423}. The algorithm is primarily obtained by applying a minimum spanning tree algorithm from \cite{KLP26}; similar ideas have appeared in other papers on Welzl orders. So our main contribution is Lemma~\ref{lem:bipbc_from_cont}, which shows how to go from a Welzl order to a biclique decomposition, and whose proof is elementary.
This approach -- using biclique decompositions -- unifies several foundational papers and several papers that have appeared over the last few years, as we explain next. 

Finally, we apply some known constructions for the Zarankiewicz problem~\cite{MR1417348, alon1999norm} to prove that Theorem~\ref{thm:main} is tight up to an $O(\log{n})$ factor for any integer $d$; see Theorem~\ref{thm:LB}.

\paragraph{Applications.}

The existence of biclique decompositions of small size has a number of interesting consequences, and Theorem~\ref{thm:main} unifies several known results. In particular, we obtain the following:
\begin{itemize}
    \item (Section~\ref{sec:zarankiewicz}) A simple proof of an upper bound on the Zarankiewicz problem for graphs of bounded VC-dimension due to Fox, Pach, Sheffer, Suk, and Zahl~\cite[Theorem 2.1]{MR3646875}, with an extra $\log{n}$ factor, but with a better dependence on the order of the forbidden~$K_{t,t}$. \item (Section~\ref{sec:multiplication}) A simple data structure for multiplying a binary matrix of VC-dimension $d$ with any vector, and for multiplying the adjacency matrix of a graph of linear neighborhood complexity with any matrix in near-quadratic time. These results recover and extend recent results of Anand, van den Brand, and McCarty~\cite{AvdBM25}, Bonnet, Giocanti, Ossona de Mendez, and Thomassé~\cite{BGOdMT23}, Kozma and Opler~\cite{KO26}, and Bonnet, Geniet, Kim and Moon~\cite{BGKM26} (in some cases with an extra $\log n$ factor).
    \item (Section~\ref{sec:apsp}) A randomized $O(n^2\log^2 n)$ all-pairs shortest paths algorithm in graphs of linear neighborhood complexity, including graphs of bounded twin-width or merge-width, without needing any contraction or construction sequence. This generalizes and simplifies recent results of Bonnet, Geniet, Kim and Moon~\cite{BGKM26}.
    \item (Section~\ref{sec:quantumcircuit}) A new upper bound of $O(n^{2-1/d}\log n)$ on the size of a stabilizer circuit constructing a graph state $\ket{G}$ of VC-dimension $d$. This new connection between quantum circuit complexity and biclique decompositions yields different proofs of upper bounds of Patel, Markov, and Hayes~\cite{optimalSynthesis}, Davies and Jena~\cite{DJ25}, and Kumabe, Mori, and Yoshimura~\cite{KMY25}. In the latter two cases we have extra $\log{n}$ factors.
    \end{itemize}

\paragraph{Graph classes with linear neighborhood complexity.} We now discuss a rich family of examples of graphs with linear or almost linear neighborhood complexity. These examples arise from a recent theory of width parameters for dense graphs which has been developed by various authors in an attempt to capture the computational complexity of first-order model-checking.

Twin-width, for instance, was introduced by Bonnet, Kim, Thomass\'e, and Watrigant~\cite{BKTW21}, and studied in an ongoing series of papers~\cite{BGKTW22,BGKTW24,BGMSTT24}.
Flip-width was introduced by Toru\'{n}czyk~\cite{T23}, and merge-width was introduced by Dreier and Toru\'{n}czyk~\cite{DT25}. Overall, these parameters are becoming more general in an attempt to capture broader classes of graphs which admit efficient algorithms for first-order model-checking. Every class of bounded twin-width has bounded merge-width, every class of bounded merge-width has bounded flip-width, and every class of bounded flip-width is monadically dependent (roughly, this means that the class has nice model-theoretic properties in the sense of Shelah~\cite{Shelah1982Book}).

It was shown recently that classes of graphs for which twin-width, merge-width, or flip-width is bounded have linear neighborhood complexity~\cite{MR4623907,BG25}. For monadically dependent graph classes, a proof has been announced that the neighborhood complexity is $m^{1+o(1)}$ where the $o(1)$ terms goes to zero as $m$ goes to infinity~\cite{monDep26}. We therefore directly obtain the following corollary of Theorem~\ref{thm:main}.

\begin{corollary}
\label{cor:width}
For any graph class of bounded twin-width, flip-width, or merge-width, every $n$-vertex graph in the class has a biclique decomposition of size $O(n\log^2 n)$. For any monadically dependent graph class, every $n$-vertex graph in the class has a biclique decomposition of size $n^{1+o(1)}$.
\end{corollary}

\noindent For the case of twin-width, this answers a question of the first and last authors~\cite{CY25}. 

We note one more type of graph class in the nearly-linear regime. It was recently proven that every hereditary small class (i.e. every graph class, which is closed under taking induced subgraphs and has at most $n! c^n$-many labeled $n$-vertex graphs) admits vertex-orderings so that the neighborhood of each vertex is the union of $O(\log^2 n)$-many intervals~\cite{MR4868423}. Thus our Lemma~\ref{lem:bipbc_from_cont} shows that the $n$-vertex graphs in such classes have biclique decompositions of size $O(n\log^3 n)$.

\paragraph{Historical remarks.}
We conclude this introduction with a brief history of biclique covers.
Biclique decompositions were studied by Chung, Erd\H{o}s, and Spencer~\cite{MR820214}, who showed that every graph on $n$ vertices has a biclique decomposition of size $O(n^2/\log n)$. The tight asymptotic upper bound $n^2/(2\log_2 n)$ was proved recently by Krapivin, Przybocki, Sanhueza-Matamala, and Subercaseaux~\cite{KPSS25}. Erd\H{o}s and Pyber~\cite{MR1452952} proved that there always exists a biclique decomposition such that every vertex is contained in $O(n/\log n)$ bicliques. 
The minimum such number for a given graph is called the \defn{bipartite degree} of the graph by Fishburn and Hammer~\cite{MR1417566}. 
Krapivin et al.~\cite{KPSS25} pointed out that the problem was also studied much earlier by soviet scientists, including Lupanov~\cite{L56} and Nechiporuk~\cite{N69}, in the context of circuit complexity. 
Biclique decompositions occur naturally in so-called \defn{confluent} drawings of graphs~\cite{DEGM05}, in which edges of bicliques are drawn as a bundle. Biclique decompositions have also been shown useful in compressing web graphs~\cite{BC08,KCA09}.

Biclique decompositions appear in many guises in computational geometry~\cite{M92,E96,MR1471987,MR1739610,BK03,cha06,AES26}.
Agarwal, Alon, Aronov, and Suri~\cite{MR1298916}, for instance, showed that visibility graphs of polygons admit compact representations in term of near-linear size biclique decompositions.
In a recent paper, two of the authors gave constructions of small-size biclique covers for semilinear graphs and terrain visibility graphs~\cite{CY25}.
It is also known from results of Do~\cite{MR4013919} and Agarwal, Aronov, Ezra, and Sharir~\cite{AAEKS25} that $d$-dimensional semialgebraic graphs admit biclique decompositions of size $O(n^{2-2/(d+1)+\varepsilon})$ for any $\varepsilon >0$. 

In a seminal paper, Feder and Motwani~\cite{FM95} observed that biclique covers of small size allowed improved agorithms for classical problems on graphs, including single-source shortest paths and maximum matching. This has also been used recently for the design of efficient geometric matching algorithms by Cabello, Cheng, Cheong, and Knauer~\cite{CCCK24}. 

There are other applications of biclique decompositions that we do not elaborate on, including data structures for independence and cut queries described by Krapivin et al.~\cite{KPSS25}, the construction of graph spanners, as observed by Conroy and T\'{o}th~\cite{CT22}, and adjacency labeling schemes, as described by Sharir and the first author~\cite{CS26}.

\section{Compact representations from small neighborhood complexity}
\label{sec:main}

In this section, we first give an upper bound on the size of a biclique decomposition as a function of a graph parameter called the \defn{contiguity}. We then proceed by recalling known results on the relation between neighborhood complexity and contiguity. Finally, we consider the problem of computing orderings that approximately realize the contiguity of a graph. We conclude by giving an efficient construction algorithm for our biclique decompositions.

\subsection{Biclique decompositions of graphs with bounded contiguity}

Let $G=(V,E)$ be a simple graph. Given a total ordering $\preceq$ on $V$ and a vertex $v\in V$, 
we let $\ell_{\preceq}(v)$ be the 
number of inclusion-wise maximal intervals in $\preceq$ such that the neighborhood of $v$ is the union of those intervals. Note that we do not consider a vertex to be in its own neighborhood.
The \defn{contiguity} of a graph $G$ is
\[
\ell(G) = \min_{\preceq} \max_v \ell_{\preceq}(v).
\]
We refer the reader to \cite[Section 2]{MR4868423} for a short history of the notion of contiguity of a graph, which stems from crossing numbers of paths in Welzl~\cite{MR1213461}.

We prove that graphs of small contiguity have biclique decompositions of small size. Note that if $\mathcal{B}$ is a biclique decomposition of an $n$-vertex graph so that every vertex is covered at most $k$ times, then the size of $\mathcal{B}$ is at most $nk$. Thus the following key lemma bounds the size of~$\mathcal{B}$. 

\begin{figure}
\centering
\begin{tikzpicture}
	\pgfmathtruncatemacro\scx{4}
	\pgfmathtruncatemacro\depth{4}
	\tkzDefPoint(0, 0){p_{0,0}}
	\tkzDefPoint(-\scx,-0.5){p_{1,-1}}
	\tkzDefPoint(\scx,-0.5){p_{1,0}}
	\tkzDrawSegment(p_{0,0},p_{1,-1})
	\tkzDrawSegment(p_{0,0},p_{1,0})
	\tkzDrawPoints(p_{0,0}, p_{1,-1}, p_{1,0})
	\foreach \i in {2, 3, ..., \depth}{
		\pgfmathtruncatemacro\a{2^(\i - 1)}
		\pgfmathtruncatemacro\b{2^(\i - 1) - 1}
		\foreach \j in {-\a, ..., \b}{
			\tkzDefPoint(\scx*(2*\j + 1)/\a,-\i/2){p_{\i,\j}}
			\pgfmathtruncatemacro\c{floor(\j/2)}
			\pgfmathtruncatemacro\d{\i - 1}
			\tkzDrawSegment(p_{\i,\j}, p_{\d,\c})
			\tkzDrawPoint(p_{\i,\j})
		}
	}
    \def \start {-5.5}
    \def \width {5}
    \def \thick {10pt}
    \draw[line width=\thick, color=red, opacity=.5] (\start-.2,-2) -- (\start+\width+.2,-2);
    \draw[ultra thick] (p_{4,-6}) -- (p_{3,-3});
    \draw[ultra thick] (p_{4,-5}) -- (p_{3,-3});
    \node[fill=black, circle, inner sep=1.5,outer sep=0] at (p_{3,-3}) {};
    \draw[ultra thick] (p_{4,-4}) -- (p_{3,-2});
    \draw[ultra thick] (p_{4,-3}) -- (p_{3,-2});
    \node[fill=black, circle, inner sep=1.5,outer sep=0] at (p_{3,-2}) {};
    \draw[ultra thick] (p_{4,-2}) -- (p_{3,-1});
    \draw[ultra thick] (p_{4,-1}) -- (p_{3,-1});
    \node[fill=black, circle, inner sep=1.5,outer sep=0] at (p_{3,-1}) {};
    \draw[ultra thick] (p_{3,-1}) -- (p_{2,-1});
    \draw[ultra thick] (p_{3,-2}) -- (p_{2,-1});
    \node[fill=black, circle, inner sep=1.5,outer sep=0] at (p_{2,-1}) {};
    \foreach \j in {-8, ..., 7}{
        \pgfmathtruncatemacro\l{\j + 9}
        \node[MyNode, label={[yshift=-.2cm, below]{\large$v_{\l}$}}] at (p_{4,\j}) {};
    }
\end{tikzpicture}
\caption{The tree $T$ and the two subtrees used to obtain the interval between $v_3$ and $v_8$.}
\label{fig:contiguity}
\end{figure}
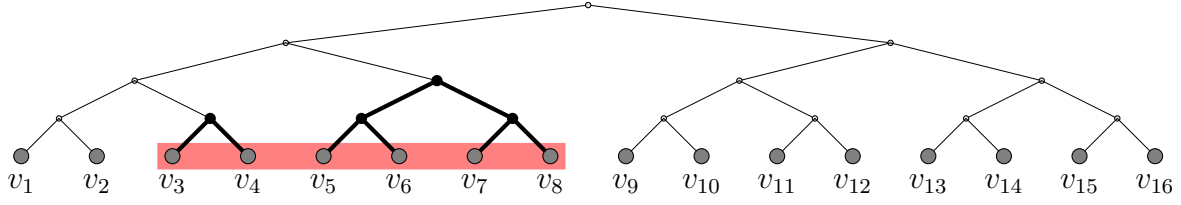

\begin{lemma}
\label{lem:bipbc_from_cont}
    Any $n$-vertex graph $G$ with contiguity $\ell$ has a biclique decomposition with at most $2n$ bicliques such that every vertex belongs to at most $(2\ell +1)(\lceil \log_2 n \rceil + 1) = O(\ell \log n)$ bicliques.
\end{lemma}
\begin{proof}
  Let $V$ denote the vertex set of $G$, and consider the optimal ordering $\preceq$ on $V$ which minimizes the contiguity. Thus the neighborhood of every vertex consists of at most $\ell$ intervals in this ordering. We construct a balanced binary tree $T$ on top of $\preceq$, the leaves of which are the vertices of $G$, and the nodes of which correspond to dyadic intervals in $\preceq$. See Figure~\ref{fig:contiguity} for a depiction. We denote by $h$ the height of $T$, defined as the number of levels in $T$. By construction, the height $h$ is at most $\lceil\log_2 n \rceil + 1$. For a node $t$ of $T$, we denote by $L(t) \subseteq V$ the set of leaves of the subtree of $T$ rooted at $t$. This is the corresponding dyadic interval. 

  Observe now that every interval $I$ in $\preceq$ is the union of at most $2h$ disjoint dyadic intervals $L(t)$. To see this, 
  observe that for each level of $T$, there are at most two nodes $t$ on that level so that $L(t)$ is a maximal dyadic interval contained in $I$. Hence the neighborhood of every vertex $v\in V$ is the union of at most $2\ell h$ disjoint dyadic intervals. 

  We define the \defn{left-neighborhood} of a vertex $v$ as the set of all neighbors $u$ of $v$ which are to the left of $v$ in the order (that is, with $u \preceq v$).
  The left-neighborhood of every vertex $v\in V$ is also the union of at most $2\ell h$ disjoint dyadic intervals.
  Now, for every dyadic interval $L(t)$, we put a biclique into our cover where the left side of the biclique is $L(t)$ and the right side of the biclique is all the vertices $v$ so that $v$ is to the right of $L(t)$ in $\preceq$, and $L(t)$ is one of the dyadic intervals used to obtain the left-neighborhood of $v$. This forms a biclique decomposition since the left-neighborhood of every vertex is partitioned by the corresponding disjoint dyadic intervals. 
  
  The number of bicliques we have is at most the number of nodes of the tree, which is at most $2n$. So it just remains to consider how many bicliques a vertex $v \in V$ belongs to. We already argued that $v$ is on the right side of at most $2\ell h$ bicliques. Moreover, $v$ is on the left side of at most $h$ bicliques since at most $h$ dyadic intervals contain~$v$. Hence each vertex is contained in at most $(2\ell +1)h$ bicliques, as claimed.
\end{proof}

In most applications of biclique decompositions, only the size of the decomposition matters.
We now define a new variant of contiguity which captures this intuition. Given an $n$-vertex graph $G=(V,E)$, its \defn{total contiguity} is
\[
t(G) = \min_{\preceq} \sum_{v \in V} \ell_{\preceq}(v),
\]
where $\preceq$ and $\ell_{\preceq}(v)$ are defined as before. The following result is now straightforward.

\begin{theorem}
\label{thm:bc_from_avg_cont}
  Any $n$-vertex graph $G$ with total contiguity $t$ has a biclique decomposition of size $O(t \log n )$.
\end{theorem}
\begin{proof}
    We may assume that $G$ has no isolated vertices and thus the total contiguity is at least $n$.
    Now, we construct a collection of bicliques as in the proof of Lemma~\ref{lem:bipbc_from_cont}. With a slightly more refined analysis, we can see that each vertex $v$ is contained in at most  $(2\ell_{\preceq}(v)+1)h$ bicliques. Hence the sum of the sizes of the bicliques is at most a constant times $\sum_{v} \ell_{\preceq}(v) \log n$, which is $O(t\log n)$, as claimed.
\end{proof}

\paragraph{Remarks.}
The idea of constructing a balanced binary tree on top of the Welzl order is reminiscent of the range-searching data structure in Welzl’s original 1988 paper~\cite[Lemma~3.2]{MR1213461} and in Chazelle and Welzl~\cite[Lemma~3.1]{MR1014739}.
We can also relate our construction to signed tree models as defined by Bonnet et al.~\cite{MR4802148}. In their terminology, we prove that graphs on $n$ vertices with contiguity $\ell$ have balanced positive tree models with degeneracy $O(\ell\log n)$.

\subsection{Bounded contiguity for graphs with small neighborhood complexity}

It is known from the seminal results of Welzl~\cite{MR1213461} that the elements of set systems of bounded VC-dimension can be linearly ordered so that each set is the union of a sublinear number of intervals in this ordering. We refer to these orders as \defn{Welzl orders}. When applied to neighborhood set systems of graphs, we can bound the contiguity of graphs of bounded VC-dimension, and apply our Lemma~\ref{lem:bipbc_from_cont}. In fact these tools also apply in the broader context of graphs of polynomial neighborhood complexity, as we now describe.

The following lemma was initially proven by Welzl~\cite{MR1213461} within a logarithmic factor; see also Chazelle and Welzl~\cite{MR1014739}. The current formulation was observed by Matou\v{s}ek, Welzl, and Wernisch~\cite{MR1262921} using results of Haussler~\cite{MR1313896}. See also Alon~\cite[Lemma 2.2]{A24}.

\begin{theorem}[Welzl orders~\cite{MR1213461,MR1262921}]
  \label{thm:VCalt}
  For any real numbers $c>0$ and $d>1$, there exists $c'=c'(c,d)$ such that the following holds.
  Any $n$-vertex graph $G$ with neighborhood complexity $\eta_G(m)\leq c\cdot m^d$ has contiguity at most $c'\cdot n^{1 - 1/d}$.
\end{theorem}

The following theorem is an extension of Theorem~\ref{thm:VCalt} to the case where $d=1$. Details are worked out in Bonnet, Duron, Sylvester, and Zamaraev~\cite[Theorem 20]{MR4868423}. 

\begin{theorem}[\cite{MR4868423}]
  \label{thm:widthcont}
  Graph classes with linear neighborhood complexity have contiguity $O(\log n)$.
\end{theorem}

\noindent Combining Lemma~\ref{lem:bipbc_from_cont} with Theorems~\ref{thm:VCalt} and~\ref{thm:widthcont}, we obtain the bounds in our Theorem~\ref{thm:main}. So it just remains to consider the algorithmic aspects.

\subsection{Algorithms}

We now give a simple approximation algorithm for the total contiguity of a graph. Since Theorem~\ref{thm:bc_from_avg_cont} applies to graphs with small total contiguity, and since every $n$-vertex graph with contiguity $\ell$ has total contiguity at most $\ell n$, this is typically enough for most applications.
 
Let $G=(V,E)$ be a graph. We will compute a minimum weight spanning tree of an associated edge-weighted clique with vertex set $V$. The weight between two vertices $u$ and $w$ is simply the size $|N(u)\triangle N(w)|$ of the symmetric difference of their neighborhoods, or, equivalently, the Hamming distance between their characteristic vectors. (The \defn{characteristic vector} of a vertex $u$ is the vector in $\{0,1\}^V$ which has a $1$ for each neighbor of $u$.) In the following lemmas, we refer to such a tree $T$ as a \defn{spanning tree on $V$}, and to its \defn{weight} as:
\[
    \sum_{uw\in E(T)} |N(u)\triangle N(w)|.
\]

First we show that graphs of small total contiguity have small-weight spanning trees on~$V$.

\begin{lemma}
\label{lem:mstub}
    For any graph $G=(V,E)$, the minimum weight of a spanning tree on $V$ is at most $2t(G)$.
\end{lemma}
\begin{proof}
    Consider an ordering $\preceq$ on $V$ that minimizes $\sum_{v} \ell_{\preceq}(v)$, and let $P$ denote the path obtained by ordering the vertices $V$ according to $\preceq$. For a given vertex $v$, we let $b_{\preceq}(v)$ be the number of \defn{breakpoints} of $N(v)$ in the ordering $\preceq$, namely, the number of edges $uw$ of $P$ such that exactly one of $u, w$ is contained in $N(v)$. Breakpoints therefore correspond to successive pairs of  bits of the form 01 or 10 in the characteristic vector of $v$. Every interval in the neighborhood $N(v)$ corresponds to at most two breakpoints, so $b_{\preceq}(v)\leq 2 \ell_{\preceq}(v)$. Hence $\sum_v b_{\preceq}(v) \leq 2 \sum_v \ell_{\preceq}(v) =2t(G)$.

    Now, by double counting, the weight of $P$ is
    \begin{eqnarray*}
    \sum_{uw\in E(P)} |N(u)\triangle N(w)|
     & = &  
     \sum_{uw\in E(P)} \sum_{v\in V} [|N(v)\cap \{u, w\}| = 1  ] \\
    & = & \sum_{v\in V} \sum_{uw\in E(P)} [ |N(v)\cap \{u, w\}| = 1 ]  \\
    & = & \sum_{v\in V} b_{\preceq}(v).
    \end{eqnarray*}
    Therefore the weight of $P$ is at most $2t(G)$, as desired.
\end{proof}

Next we show the other direction; given a spanning tree of small weight, we can compute an order which shows that $G$ has small total contiguity. This is a simple application of the double-tree shortcutting method for the metric traveling salesman problem.

\begin{lemma}
\label{lem:mstlb}
    Given a graph $G=(V,E)$ with $n$ vertices and a spanning tree $T$ on $V$ with weight $W$, one can in time $O(n)$ compute an ordering $\preceq$ on $V$ such that
    $\sum_{v} \ell_{\preceq}(v) \leq W + n$.
\end{lemma}
\begin{proof}
    A depth-first traversal of $T$ uses every edge of $T$  exactly twice. Hence the sum of the successive weights of the edges in a traversal is exactly $2W$.
    Since the Hamming distance satisfies the triangle inequality, we can shortcut such a traversal without increasing the total weight.
    Let $\preceq$ order the vertices of $G$ by when they are first visited by a depth-first traversal. Also let $P$ denote the path on $V$ obtained by ordering $V$ according to $\preceq$. From our observation about shortcuts, 
    \[
    \sum_{uw \in E(P)}|N(u)\Delta N(w)|\leq 2W.
    \]
    Now, as argued in Lemma~\ref{lem:mstub}, the weight of $P$ equals $\sum_{v} b_{\preceq}(v)$.
    Since $\ell_{\preceq}(v) \leq b_{\preceq}(v)/2 + 1$, we obtain
    \[
    \sum_{v\in V} \ell_{\preceq}(v)\leq n + \sum_{v\in V} b_{\preceq}(v)/2 \leq W + n. 
    \] 
\end{proof}

We now use a recent randomized algorithm of Kowaluk, Lingas, and Persson~\cite[Theorem~3]{KLP26} to find an approximate minimum weight spanning tree on $n$ binary vectors of size $n$ with respect to the Hamming distance in time $O(n^2\log n)$. 

\begin{theorem}[\cite{KLP26}]
\label{thm:mstalg}
For any positive $\epsilon < 1/2$, one can compute a $\frac{1+\epsilon}{1-\epsilon}$-approximate minimum weight spanning tree on a set of $n$ points in $\{0,1\}^n$ with respect to the Hamming distance in time $O(n^2\log n/\epsilon^2)$ with high probability.
\end{theorem}

For our purposes, we can fix for instance $\epsilon = 1/3$ to obtain a $2$-approximation. We now combine all of these results to obtain the following approximation algorithm for the total contiguity.

\begin{lemma}
\label{lem:ordalg}
There is an $O(n^2\log n)$-time randomized algorithm that given an $n$-vertex graph $G=(V,E)$, computes an ordering $\preceq$ of $V$ such that $\sum_{v} \ell_{\preceq}(v)\leq 4 t(G) + n$.
\end{lemma}
\begin{proof}
    Run the algorithm of Theorem~\ref{thm:mstalg} to obtain a 2-approximate minimum weight spanning tree on $V$, say of weight $W'$. Thus
    $W'\leq 2W$, where $W$ is the minimum weight of a spanning tree on $V$.
    From Lemmas~\ref{lem:mstub} and~\ref{lem:mstlb}, one can compute an ordering $\preceq$ on $V$ such that 
    \begin{align*}
    \sum_{v \in V} \ell_{\preceq}(v) & \leq  W' + n && \text{(from Lemma~\ref{lem:mstlb})} \\
    & \leq  2W + n&& \text{(from the 2-approximation})\\
    & \leq  4 t(G) + n && \text{(from Lemma~\ref{lem:mstub})}.
    \end{align*}
    This completes the proof of Lemma~\ref{lem:ordalg}.
\end{proof}

Finally, we show how to construct the biclique decomposition from the ordering. Note that the algorithm and its running time are independent of the total contiguity~$t(G)$.
 
\begin{theorem}
\label{thm:alg}
    Given an $n$-vertex graph $G=(V,E)$, there is an $O(n^2\log n)$-time randomized algorithm that computes a biclique decomposition of $G$ of size $O(t(G)\cdot \log n)$. 
\end{theorem}
\begin{proof}
    First we apply Lemma~\ref{lem:ordalg} to find an ordering $\preceq$ on $V$ with $\sum_{v} \ell_{\preceq}(v)\leq 4 t(G) + n$ in time $O(n^2\log n)$. We can construct a binary search tree on this ordering in time $O(n\log n)$. Then for each vertex $v$, we need to identify the dyadic intervals whose union equals $N(v)$. This takes time $O(n^2)$, by proceeding as follows. For each vertex $v$, we label each node of the tree by $0/1/?$ depending on whether $v$ is adjacent to none/all/some of the vertices in the corresponding dyadic interval. For a fixed vertex $v$, this can be done in time $O(n)$ by working bottom-up on $T$. Then we can find the maximal nodes labeled $1$ in time $O(n)$ by working top-down.
\end{proof}

Combining Theorem~\ref{thm:alg} with Theorems~\ref{thm:VCalt} and \ref{thm:widthcont} proves Theorem~\ref{thm:main}.

\paragraph{Remarks.}

Finding a Welzl order via a spanning tree can also be traced back to Welzl's original paper~\cite[Lemma 3.3]{MR1213461}. 
The classical algorithm for computing Welzl orders involves a well-known reweighting method~\cite{MR1213461}. For a detailed analysis, we refer to Matou\v{s}ek's description of the method~\cite[Lemma 5.17]{MR2683232}, and its adaptation to the case of arbitrary (including linear) shatter functions due to Bonnet, Duron, Sylvester, and Zamaraev~\cite[Appendix A]{MR4868423}. 
Recently, Dreier and Kuske~\cite{DK26} proposed an $O((n + m)\log n)$-time algorithm for computing Welzl orders with contiguity $O(\log^2 n)$ for graphs of linear neighborhood complexity on $n$ vertices and $m$ edges. For graphs of neighborhood complexity $\eta_G(m)=O(m^d)$, they obtain orderings with contiguity $O(n^{1-1/d^2}\log^2 n)$.
Minimum spanning trees with respect to the Hamming distance have also been used recently in results about matrix multiplication, such as in Anand et al.~\cite{AvdBM25} and Kozma and Opler~\cite{KO26}, and on diameter computations by Duraj, Konieczny, and Potepa~\cite[Lemma 13]{DKP24}. We do not claim any novelty here. One can also obtain orderings achieving a total contiguity within a polylogarithmic factor of the bound of Theorem~\ref{thm:VCalt} in time $O(n^{1+1/d}\log n)$, see for instance Chan, Chang, Gao, Kisfaludi-Bak, Lei, and Zheng~\cite[Lemma 7.4]{CCGKLZ25}, and Bonnet, Geniet, Kim, and Moon~\cite[Lemma 10]{BGKM26}. Note that for $d=1$, our main use case, the running time matches that of our algorithm, which also provides a better approximation.

\section{Applications}

We now proceed with applications of the above results to the Zarankiewicz problem, matrix multiplication, finding shortest paths and quantum circuit complexity.

\subsection{Zarankiewicz Problem}
\label{sec:zarankiewicz}

Zarankiewicz proposed the following problem in the 1950s: Given a graph on $n$ vertices, which does not contain any complete bipartite graph $K_{t, t}$ as subgraph, what is its maximum number of edges? The well-known K\"ovari-S\'os-Tur\'an theorem~\cite{MR65617} gives an upper bound. We refer to the recent survey of Smorodinsky~\cite{S24} for more background on this family of problems in geometric contexts.
The following simple observation gives an upper bound for the Zarankiewicz problem as a function of the size of a biclique cover. 

\begin{lemma}\label{lemma:sizeviaZarankiewicz}
    Let $G$ be a graph without a $K_{d,t}$ subgraph, for some $d,t\in\mathbb{N}$ where $t\ge d$, and with a biclique cover of size $s$. Then $G$ has at most $t\cdot s$ edges.
\end{lemma}
\begin{proof}
Indeed, for any biclique $B$ in the cover, we have $|E (B)|\leq t\cdot |V (B)|$ since every vertex on the larger side of $B$ has at most $t$ neighbors. Hence, $|E(G)|\leq \sum_B |E(B)|\leq t\cdot\sum_B |V(B)| = t\cdot s$.
\end{proof}

In particular, we can apply Lemma~\ref{lem:bipbc_from_cont} and Theorem~\ref{thm:VCalt} to recover, up to a logarithmic factor, a result from Fox, Pach, Sheffer, Suk, and Zahl~\cite[Theorem 2.1]{MR3646875} for graphs of bounded VC-dimension or polynomial neighborhood complexity.

\begin{corollary}
    For any real numbers $c>0$ and $d>1$, there exists $c'=c'(c, d)$ such that the following holds.
    Let $G$ be a graph on $n$ vertices with neighborhood complexity $\eta_G(m)\leq c\cdot m^d$, and without a $K_{t,t}$ subgraph, for some $t\in\mathbb{N}$.
    Then $G$ has at most $c'\cdot t\cdot n^{2 - 1/d}\log n$ edges.
\end{corollary}

Like the proof of the existence of Welzl orders of Theorem~\ref{thm:VCalt}, the proof of Fox et al.~\cite{MR3646875} relies on Haussler's packing lemma~\cite{MR1313896}, but their bound has a hidden constant factor of $\Omega(t^{2d + 1})$. Hence we obtain an improved bound in the regime where $t>\log^{\frac1{2d}} n$.
The exponent $2-1/d$ can be shown to be best possible, as discussed in the conclusion of Fox et al.~\cite{MR3646875}, using the constructions of Koll\'{a}r, R\'{o}nyai, and Szab\'{o}~\cite{MR1417348} and Alon, R{\'o}nyai, and Szab{\'o}~\cite{alon1999norm}. This implies a lower bound on the size of biclique decompositions. To see this, first consider the following theorem. We note that if $G$ is a graph with no $K_{d,t}$ -subgraph, then the bipartite graph whose bipartite adjacency matrix is the same as the adjacency matrix of $G$ also has no $K_{d,t}$-subgraph. Thus any construction can be turned into a bipartite construction.

\begin{theorem}[\cite{alon1999norm}]\label{thm:LBZarankiewicz}
    For any fixed integers $d\ge 2$ and $t\ge (d-1)!+1$, there exists a bipartite graph with both parts of size $n$, without a $K_{d,t}$ subgraph, and with $\Omega(n^{2-1/d})$ edges.  
\end{theorem}

Combining the above results we can deduce the following lower bound.

\begin{restatable}{theorem}{lowerbound}\label{thm:LB}
    For any fixed positive integer $d$, there exists an integer $c$ so that there are graphs with neighborhood complexity $\eta_G(m)\leq cm^d$ where the size of any biclique cover is $\Omega(n^{2-1/d})$. 
\end{restatable}

\begin{proof}
    Let $G$ be the graph from \Cref{thm:LBZarankiewicz} with $t=(d-1)!+1$. Similarly to the discussion in Fox et al.~\cite{MR3646875} and Alon, Moran, and Yehudayoff~\cite[Section 4.1.1]{AMY14}, the neighborhood complexity in such a graph is at most $O(m^d)$. Indeed, let $C$ be a subset of size $m$ of vertices in one of the parts of the bipartition of $G$. The number of different neighborhoods in $C$ of size at most $(d-1)$ is at most $\binom{m}{d-1}+\binom{m}{d-2}+\dots+\binom{m}{0}$. The number of neighborhoods in $C$ of size at least $d$ can be bounded by $(t-1)\binom{m}{d}$. The reason is that any set of $d$ vertices can be contained in at most $(t-1)$ neighborhoods as otherwise we get a copy of $K_{d,t}$. 

    Now assume that $G$ has a biclique cover of size $s$, then by \Cref{lemma:sizeviaZarankiewicz}, we get that $s=\Omega(n^{2-1/d})$, as required. 
\end{proof}

\subsection{Matrix multiplication}
\label{sec:multiplication}

In this subsection we describe how to deduce an efficient algorithm for the multiplication of two matrices where {\it only} one of the matrices needs to be a $0/1$-matrix and have a restricted structure. Variants of this question were studied by Anand, van den Brand, and McCarty in~\cite{AvdBM25} for matrices with bounded VC-dimension, and by Kozma and Opler~\cite{KO26} and Bonnet, Geniet, Kim and Moon~\cite{BGKM26}, generalizing on a result by Bonnet, Giocanti, Ossona de Mendez, and Thomassé in~\cite{BGOdMT23}, for matrices with bounded twin-width. Comparing with the result by Bonnet et al.~\cite{BGKM26}, the running time of our algorithm has an additional $\log n$, but on the other hand, can be applied to more general graph classes than just adjacency matrices of graphs with bounded twin-width.

We refer to Anand et al.~\cite{AvdBM25} and Karande, Chellapilla, and Andersen~\cite{KCA09} for applications of this fast multiplication algorithm to random walks, the PageRank algorithm, graph neural networks, and the analysis of web graphs, among others.

\begin{theorem}
\label{thm:mult}
    Given a graph $G$ with linear neighborhood complexity and adjacency matrix $M\in \{0,1\}^{n\times n}$ and another arbitrary matrix $P\in \mathbb{R}^{n\times n}$, there is a randomized algorithm for computing the product $M P$ in time $O(n^2\log^2 n)$.
\end{theorem}

\begin{proof}
One can observe that a biclique in $G$ corresponds to a collection of rows and columns in $M$ such that the sub-matrix of $M$ on the intersection of those rows and columns is the identically-$1$ matrix. Hence a biclique decomposition of $G$ corresponds to a decomposition of the $1$-entries of $M$ into all $1$s sub-matrices. 
From Theorem~\ref{thm:main}, such a biclique decomposition of size $O(n\log^2 n)$ can be computed in time $O(n^2\log n)$.

The multiplication of two $n\times n$ matrices can be done by $n$ multiplications of a matrix with a vector. Hence we focus on the latter.  
Computing $Mv$ where $M$ is a $a\times b$ all ones matrix with a vector $v\in \mathbb{R}^n$ can be done in time $O(a+b)$. Therefore given a biclique decomposition of the underlying graph of $M$ of size $s$, the product $Mv$ can be computed in time $O(s)$. The whole product $MP$ can therefore be computed in time $O(ns)=O(n^2\log^2 n)$.
\end{proof}
\begin{figure}[h!]
\begin{lstlisting}[language=Python]
def multiply(B: biclique_decomposition, v):
     result = [0] * len(v)
     for left, right in B:
         s = sum(v[j] for j in left)
         for j in right: result[j] += s
     return result
\end{lstlisting}
    \caption{\label{fig:pythonmult}A Python code snippet implementing the algorithm described in the proof of Theorem~\ref{thm:mult} for multiplying the adjacency matrix of a graph with a vector, given a biclique decomposition of the graph.}
\end{figure}

Given a biclique decomposition, the multiplication algorithm can be implemented easily, as illustrated in Figure~\ref{fig:pythonmult}. 
Note that we need not restrict to symmetric matrices: We can also consider the biadjacency matrix of a bipartite graph $G=(A\sqcup B, E)$, and the neighborhood set system $\{N(v): v\in A\}$ on the ground set $B$.
Also note that our result also applies to almost-linear neighborhood complexity, hence monadically dependent graph classes, with the running time $n^{2+o(1)}$.

\subsection{Shortest paths}
\label{sec:apsp}

Feder and Motwani~\cite{FM95} explain how to compute breadth-first search trees on compact representations of graphs. We give an elementary proof.

\begin{lemma}[\cite{FM95}]
\label{lem:sssp}
Given a biclique cover of size $s$ of a graph $G$, the breadth-first search tree of any vertex of $G$ can be computed in time $O(s)$.
\end{lemma}
\begin{proof}
One can encode the biclique cover as a directed graph as follows: Replace every biclique $A\times B$ in the cover by a directed graph involving two additional vertices, say $\alpha$ and $\beta$, and containing all directed edges from $A$ to $\alpha$ and from $\alpha$ to $B$, and also from $B$ to $\beta$ and from $\beta$ to $A$. We obtain a bipartite directed graph $D$ such that the distance between two vertices from $V(G)$ in $D$ is exactly twice their distance in $G$. By definition, the number of directed edges of $D$ is exactly equal to $2s$. Hence we can run the standard breadth-first search algorithm on $D$ in time $O(s)$ and recover a breadth-first search tree on $G$.
\end{proof}

By first computing a biclique cover using Theorem~\ref{thm:alg}, and then running a breadth-first search from every vertex using Lemma~\ref{lem:sssp}, we obtain the following.

\begin{theorem}
    There is a randomized algorithm running in time $O(n^2\log^2 n)$ for the all-pairs shortest paths problem on graphs of linear neighborhood complexity.
\end{theorem}

This generalizes a result of Bonnet, Geniet, Kim, and Moon~\cite{BGKM26} for shortest paths on graphs of bounded twin-width. It also applies to graphs of bounded merge-width, while the algorithm described in~\cite{BGKM26} requires a witness of low merge-width. 

\subsection{Quantum circuit complexity}
\label{sec:quantumcircuit}

In this section we prove Theorem~\ref{thm:quantumCircuit}, which says that every $n$-vertex graph state $\ket{G}$ with a biclique decomposition of size $k$ can be constructed by a stabilizer circuit of size $O(k+n)$. That is, $\ket{G}$ can be built from the state $\ket{0}^{\otimes n}$ using at most $O(k+n)$ gates, each of which is an $H$, $S$, or $CNOT$ gate. See the next subsection for a complete introduction. For now we discuss how to use this theorem to obtain some other previously known results in quantum circuit complexity. 

First of all, since every $n$-vertex graph has a biclique decomposition of size $O(n^2/\log n)$ due to a 1983 theorem of Chung, Erd\H{o}s, and Spencer~\cite{MR820214}, we recover the circuit complexity upper bound of~\cite{improvedSimulation04, optimalSynthesis} (which says that every $n$-vertex graph can be built from a stabilizer circuit of size $O(n^2/\log{n})$).  This upper bound matches the information-theoretic lower bound of $\Omega(n^2/\log n)$ from~\cite[Corollary~9]{circuitLowerBound}. Note that our theorem provides an improvement for graphs with biclique decompositions of size $o(n^2/\log n)$, so in particular for graphs with polynomial neighborhood complexity (see Theorem~\ref{thm:main}). We will also recover some quantum circuit complexity upper bounds from~\cite{DJ25, KMY25} up to logarithmic factors; see Lemma~\ref{lem:intervalCircle}. 

Finally, since every graph with average degree $d$ has a biclique decomposition of size $O(nd)$, our biclique-based method generalizes the degeneracy-based method (which is useful in practice for quantum error-correcting codes, as explained in~\cite{fastSimulationGS06}). We note that for graphs with small total contiguity, Theorem~\ref{thm:alg} shows how to efficiently construct the biclique decomposition and thus, as we will show, also the stabilizer circuit. So many of these results are effective.

Graph states have broad applications across quantum computing. Some of the most well-known uses arise in quantum error-correcting codes~\cite{CliffordErrorCorrection} and a ``resource-based'' model of quantum computing known as \defn{measurement-based quantum computing}~\cite{OneWayQC}. These applications are surveyed in~\cite{graphStateEntanglement}. More recently, there is renewed interest in graph states for applications in quantum networking. Graph states naturally model the \defn{multiparty} setting, in which many distributed users share a single quantum state. See for instance~\cite{transformingStates, quantumNetworkRoutingAndLC, OtherDistributing} for work in this area as well as~\cite{graphStatesSecretSharing} for connections to quantum secret-sharing.

We note that one reason stabilizer circuits are interesting is that they can be simulated by an efficient classical algorithm; this is called the Gottesman-Knill Theorem~\cite{Gottesman98}. Several of these results about quantum circuit complexity were obtained as improvements to that algorithm~\cite{improvedSimulation04, fastSimulationGS06}. Overall, there are many reasons one might want an efficient quantum algorithm to construct a graph state, arising from the search for a practical quantum architecture. We refer the reader to~\cite{CLXLLH25, NGClifford2024, minimizingLCEquiv, JWFL26} for other interesting heuristic-based approaches.

\paragraph{Fundamentals.}
While graph states are more naturally introduced as \defn{stabilizer states}, we opt to give a simpler, more explicit introduction to make the paper more accessible. We also introduce some of the basics of quantum computing for convenience.

First of all, a \defn{qubit} is a complex vector in $\mathbb{C}^2$ of norm~$1$; this is the basic unit of computation. We fix an orthonormal basis of $\mathbb{C}^2$ known as the \defn{computational basis} for that qubit. In particular, we define $\ket{0} \coloneqq (1,0)$ and $\ket{1} \coloneqq (0,1)$. An \defn{$n$-qubit quantum state} is basically just a unit vector in $\mathbb{C}^{2^n}$, however, we need to keep track of the qubits (this is done by using Hilbert spaces). 

To do this we define \defn{tensor products}. For two qubits $(\alpha,\beta)$ and $(a,b)$, their tensor product is $(\alpha,\beta) \otimes (a,b)=(\alpha a, \alpha b, \beta a, \beta b)$. In general, when we take the tensor product of two vectors or matrices, their dimensions multiply. It is convenient to introduce some shorthand notation; for instance we write $\ket{01}\coloneqq\ket{0} \otimes \ket{1}=(1,0) \otimes (0,1)=(0, 1, 0, 0)$. So really we are just using binary strings of length $n$ to index entries in $\mathbb{C}^{2^n}$. Thus $\{\ket{s}: s \in \{0,1\}^n\}$ is an orthonormal basis for $\mathbb{C}^{2^n}$. Since we consider $\mathbb{C}^{2^n}$ as a tensor product of $n$ qubits, we write $\mathbb{C}^{\otimes n}$ instead.

Now, given an $n$-vertex graph $G=(V,E)$, we can define its corresponding \defn{graph state} $\ket{G} \in \mathbb{C}^{\otimes V}$. For each binary string $s \in \{0,1\}^V$, let $G_s$ be the subgraph of $G$ induced on the support of $s$. That is, we restrict our attention to the vertices of $G$ which index a $1$ in the string $s$ (and to the edges between these vertices). Let $|G_s|$ be the number of edges in $G_s$. We define\begin{align*}
\ket{G} = \frac{1}{\sqrt{2}^{n}}\sum_{s \in \{0,1\}^V} (-1)^{|G_s|} \ket{s}.
\end{align*}\noindent So we can think of $\ket{G}$ as a vector in $\mathbb{C}^{2^V}$ where each entry is either positive or negative $1/\sqrt{2}^{n}$ depending on the parity of the number of edges between the corresponding vertices.

A \defn{stabilizer circuit} is a circuit which uses only the Hadamard gate $H$, the phase gate $S$, and the controlled-$Z$ gate $CZ$, depicted below. The \defn{size} of a stabilizer circuit is the number of gates it contains. (Often $CNOT$ is used instead of $CZ$, but these definitions are equivalent up to a small multiplicative factor in the size of the circuit. Sometimes it is more important to minimize one type of gate than another; for instance $CZ$ gates might be considered more expensive since they act on two qubits and create entanglement. However, our theorem uses roughly the same number of $CZ$ gates as other gates. So for our purposes, we might as well bound the total number of gates.)

\begin{align*}
H=\frac{1}{\sqrt{2}}\begin{bmatrix}
1 & 1 \\
1 & -1 
\end{bmatrix}, { \hskip 1cm} S=\begin{bmatrix}
1 & 0 \\
0 & i
\end{bmatrix}, { \hskip 1cm} CZ=\begin{bmatrix}
1 & 0 & 0 & 0\\
0 & 1 & 0 & 0\\
0 & 0 & 1 & 0\\
0 & 0 & 0 & -1
\end{bmatrix}
\end{align*}
\noindent See Figure~\ref{fig:circuit} for an example. When a quantum gate only acts on qubits $u$ and $v$, for instance, it means that the action of the gate on $\ket{s}$ is determined by only the entries at the $u$ and $v$ positions. Thus, for instance, if we perform the $CZ$ gate on qubits $u$ and $v$, we write this gate as $CZ_{uv}$, and $CZ_{uv}\ket{s}$ is either $\ket{s}$ or $-\ket{s}$, where it is the latter if there is a $1$ at the positions of both $u$ and~$v$.

\begin{figure}
\centering

\begin{tikzpicture}[x=1.15cm,y=0.9cm]

  \draw (0,0) -- (3.5,0);
  \draw (0,-1) -- (3.5,-1);

  \node[left] at (0,0) {$\ket{0}$};
  \node[left] at (0,-1) {$\ket{0}$};

  \node[
    draw,
    fill=white,
    minimum width=6mm,
    minimum height=6mm,
    inner sep=0pt
  ] at (1,0) {$H$};

  \node[
    draw,
    fill=white,
    minimum width=6mm,
    minimum height=6mm,
    inner sep=0pt
  ] at (1,-1) {$H$};

  \fill (2,0) circle (1.7pt);
  \draw (2,0) -- (2,-1);

  \node[
    draw,
    fill=white,
    minimum width=6mm,
    minimum height=6mm,
    inner sep=0pt
  ] at (2,-1) {$Z$};

  \node at (3.78,-0.5) {%
    $\left.
      \vphantom{\rule[-0.75cm]{0pt}{1.5cm}}
    \right\}$%
  };

  \node[right] at (4.00,-0.5) {$\ket{G}$};

\end{tikzpicture}

\caption{A stabilizer circuit which constructs the graph with two vertices and one edge.}
\label{fig:circuit}
\end{figure}
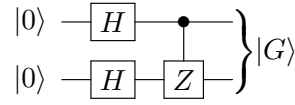

We can equivalently write a graph state as $\ket{G} = \prod_{uv\in E} CZ_{uv} \ket{+}^{\otimes V}$. That is, first we perform a Hadamard gate on each qubit in order to put it in the state $\ket{+}=H\ket{0}=(\ket{0}+\ket{1})/\sqrt{2}$. Then for each edge $uv$, we perform a controlled-$Z$ gate on qubits $u$ and $v$. So we naturally have a stabilizer circuit of size $n+m$ that constructs a graph state $\ket{G}$ with $n$ vertices and $m$ edges from $\ket{0}^{\otimes V}$. Our goal is to beat this bound for graphs with many edges using biclique decompositions.

\paragraph{Graph operations.}
In order to define the stabilizer circuit in a convenient way, we need to introduce certain graph operations. The connection between these operations and quantum computing was first discovered by Van den Nest, Dehaene, and De Moor~\cite{LCEquiv}. 

First of all, \defn{locally complementing} at a vertex $v$ of a graph $G$ replaces the induced subgraph on the neighborhood of $v$ by its complement. We write $G*v$ for the graph obtained from $G$ by locally complementing at $v$. So for any pair of neighbors $a$ and $b$ of $v$ in $G$, we ``switch'' adjacency between $a$ and~$b$. The next lemma follows from~\cite{LCEquiv}, and we omit the proof. 

\begin{lemma}[see~\cite{LCEquiv}]
\label{lem:localComplStabilizer}
    For any graph state $\ket{G}$ and any vertex $v$ of $G$, the graph state $\ket{G*v}$ can be obtained from $\ket{G}$ by performing the gate $HSH$ on the qubit of $v$ and then performing the gate $S^3=S^{\dagger}$ on each neighbor of $v$ in $G$.
\end{lemma}

Thus, if $v$ has $k$ neighbors in $G$, then there is a stabilizer circuit of size $O(k)$ which builds $\ket{G*v}$ from $\ket{G}$. We now need one more operation. For an edge $uv$ of a graph $G$, \defn{pivoting} at $uv$ results in the graph $G*u*v*u$; this
operation is well-defined since $G*u*v*u = G*v*u*v$ by~\cite[Corollary~2.2]{RWAndVM}. We denote
this new graph by $G \times uv$. Equivalently, $G \times uv$ is obtained from $G$ by switching the labels of $u$ and $v$, and then switching the adjacency between any pair of vertices $a$ and $b$ in $(N(u) \cup N(v) )\setminus \{u,v\}$ which have an odd number of common neighbors in $\{u,v\}$ (see~\cite[Proposition~2.1]{RWAndVM}). See Figure~\ref{fig:pivot} for a depiction of this operation.

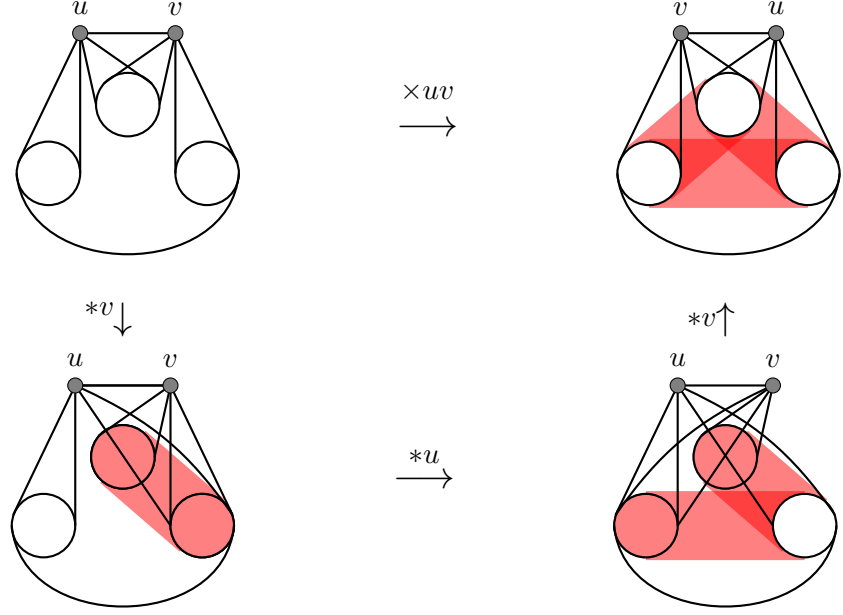
\begin{figure}
\centering
\begin{tikzpicture}[scale=.7, every node/.style={MyNode}]
    \def \r {.6}
    \def \x {1.5}
    \def \y {.9}
    \def \start {90}
    \node[label=above:{$u$}] (u) at (-0.9, 1.75) {};
    \node[label=above:{$v$}] (v) at (0.9, 1.75) {};
    \draw[thick, fill=white] (0,.4) circle (\r) {};
    \draw[thick, fill=white] (-\x, -\y) circle (\r) {};
    \draw[thick, fill=white] (\x, -\y) circle (\r) {};
    \draw[thick] (u) -- (v);
    \draw[thick] (u) -- (-\x-\r+.025, -\y+.2);
    \draw[thick] (u) -- (-\x+\r, -\y);
    \draw[thick] (v) -- (\x-\r, -\y);
    \draw[thick] (v) -- (\x+\r-.025, -\y+.2);
    \draw[thick] (u) -- (-\r, 0.4);
    \draw[thick] (u) -- (\r-.145, 0.8);
    \draw[thick] (v) -- (-\r+.145, 0.8);
    \draw[thick] (v) -- (\r, 0.4);
    \draw[thick] (-\x-\r, -\y) to [bend right=85, looseness=1.25] (\x+\r, -\y);
\end{tikzpicture}\hskip 2cm
\begin{tikzpicture}[scale=.9, every node/.style={MyNode}]
    \node[draw=none, fill=none] at (0, -2) {};
    \node[rectangle, draw=none, fill=none, label={[yshift=-.15cm, above]{$\times uv$}}] (center) at (0, .15) {\Large$\longrightarrow$};
\end{tikzpicture}\hskip 2cm
\begin{tikzpicture}[scale = .7, every node/.style={MyNode}]
    \def \r {.6}
    \def \x {1.5}
    \def \y {.9}
    \def \thick {26pt} 
    \def \start {90}
    \node[label=above:{$v$}] (u) at (-0.9, 1.75) {};
    \node[label=above:{$u$}] (v) at (0.9, 1.75) {};
    \draw[line width=\thick, color=red, opacity=.5] (0,.4) -- (-\x, -\y);
    \draw[line width=\thick, color=red, opacity=.5] (0,.4) -- (\x, -\y);
    \draw[line width=\thick, color=red, opacity=.5] (-\x, -\y) -- (\x, -\y);
    \draw[thick, fill=white] (0,.4) circle (\r) {};
    \draw[thick, fill=white] (-\x, -\y) circle (\r) {};
    \draw[thick, fill=white] (\x, -\y) circle (\r) {};
    \draw[thick] (u) -- (v);
    \draw[thick] (u) -- (-\x-\r+.025, -\y+.2);
    \draw[thick] (u) -- (-\x+\r, -\y);
    \draw[thick] (v) -- (\x-\r, -\y);
    \draw[thick] (v) -- (\x+\r-.025, -\y+.2);
    \draw[thick] (u) -- (-\r, 0.4);
    \draw[thick] (u) -- (\r-.145, 0.8);
    \draw[thick] (v) -- (-\r+.145, 0.8);
    \draw[thick] (v) -- (\r, 0.4);
    \draw[thick] (-\x-\r, -\y) to [bend right=85, looseness=1.25] (\x+\r, -\y);
\end{tikzpicture}\\
\begin{tikzpicture}[scale=.7, every node/.style={MyNode}]
    \def \r {.6}
    \def \x {1.5}
    \def \y {.9}
    \def \start {90}
    \node[rectangle, draw=none, fill=none, label={[yshift=-.15cm, left]{$*v$}}] (center) at (0, 3) {\Large$\downarrow$};
    \node[label=above:{$u$}] (u) at (-0.9, 1.75) {};
    \node[label=above:{$v$}] (v) at (0.9, 1.75) {};
    \draw[thick, fill=white] (0,.4) circle (\r) {};
    \draw[thick, fill=white] (-\x, -\y) circle (\r) {};
    \draw[thick, fill=white] (\x, -\y) circle (\r) {};
    \draw[thick] (u) -- (v);
    \def \thick {26pt} 
    \draw[line width=\thick, color=red, opacity=.5] (0,.4) -- (\x, -\y);
    \draw[thick, fill=white] (0,.4) circle (\r) {};
    \draw[thick, fill=red, opacity=.5] (0,.4) circle (\r) {};
    \draw[thick, fill=white] (\x, -\y) circle (\r) {};
    \draw[thick, fill=red, opacity=.5] (\x, -\y) circle (\r) {};
    \draw[thick] (u) -- (v);
    \draw[thick] (u) -- (-\x-\r+.025, -\y+.2);
    \draw[thick] (u) -- (-\x+\r, -\y);
    \draw[thick] (v) -- (\x-\r, -\y);
    \draw[thick] (v) -- (\x+\r-.025, -\y+.2);
    \draw[thick] (u) -- (\x-\r, -\y);
    \draw[thick] (u) to [bend left=15] (\x+\r-.025, -\y+.2);
    \draw[thick] (v) -- (-\r+.145, 0.8);
    \draw[thick] (v) -- (\r, 0.4);
    \draw[thick] (-\x-\r, -\y) to [bend right=85, looseness=1.25] (\x+\r, -\y);
\end{tikzpicture}\hskip 2cm
\begin{tikzpicture}[scale=.9, every node/.style={MyNode}]
    \node[draw=none, fill=none] at (0, -2) {};
    \node[rectangle, draw=none, fill=none, label={[yshift=-.15cm, above]{$*u$}}] (center) at (0, .15) {\Large$\longrightarrow$};
\end{tikzpicture}\hskip 2cm
\begin{tikzpicture}[scale = .7, every node/.style={MyNode}]
    \def \r {.6}
    \def \x {1.5}
    \def \y {.9}
    \def \thick {26pt} 
    \def \start {90}
    \node[rectangle, draw=none, fill=none, label={[yshift=-.3cm, left]{$*v$}}] (center) at (0, 3) {\Large$\uparrow$};
    \node[label=above:{$u$}] (u) at (-0.9, 1.75) {};
    \node[label=above:{$v$}] (v) at (0.9, 1.75) {};
    \draw[line width=\thick, color=red, opacity=.5] (0,.4) -- (\x, -\y);
    \draw[line width=\thick, color=red, opacity=.5] (-\x, -\y) -- (\x, -\y);
    \draw[thick, fill=white] (0,.4) circle (\r) {};
    \draw[thick, fill=red, opacity=.5] (0,.4) circle (\r) {};
    \draw[thick, fill=white] (-\x, -\y) circle (\r) {};
    \draw[thick, fill=red, opacity=.5] (-\x, -\y) circle (\r) {};
    \draw[thick, fill=white] (\x, -\y) circle (\r) {};
    \draw[thick] (u) -- (v);
    \draw[thick] (u) -- (-\x-\r+.025, -\y+.2);
    \draw[thick] (u) -- (-\x+\r, -\y);
    \draw[thick] (v) -- (-\r+.145, 0.8);
    \draw[thick] (v) -- (\r, 0.4);
    \draw[thick] (u) -- (\x-\r, -\y);
    \draw[thick] (u) to [bend left=15] (\x+\r-.025, -\y+.2);
    \draw[thick] (v) to [bend right=15] (-\x-\r+.025, -\y+.2);
    \draw[thick] (v) -- (-\x+\r, -\y);
    \draw[thick] (-\x-\r, -\y) to [bend right=85, looseness=1.25] (\x+\r, -\y);
\end{tikzpicture}
\vspace{-.25cm}
\caption{A depiction of pivoting with the ``switched'' edges/non-edges highlighted in red. The three circles represent the sets $N(u)\setminus N(v) - v$, $N(u)\cap N(v)$, and $N(v)\setminus N(u) - u$, respectively.}
\label{fig:pivot}
\end{figure}

By considering pivoting as a sequence of three local complementations as depicted in Figure~\ref{fig:pivot}, and by applying Lemma~\ref{lem:localComplStabilizer}, we obtain the following lemma. 

\begin{lemma}
\label{lem:pivotStabilizer}
    For any graph state $\ket{G}$ and any edge $uv$ of $G$ so that $u$ has $k_1$ neighbors and $v$ has $k_2$ neighbors, the graph state $\ket{G\times uv}$ can be obtained from $\ket{G}$ by a stabilizer circuit of size $O(k_1+k_2)$.
\end{lemma}

\paragraph{The main result.}
Now we show how to construct small stabilizer circuits for graphs with small biclique decompositions. The proof of Theorem~\ref{thm:quantumCircuit} is effective; it yields an efficient classical algorithm to construct the desired stabilizer circuit given the graph $G$ and its biclique decomposition~$\mathcal{B}$.

\begin{theorem}
\label{thm:quantumCircuit}
Every $n$-qubit graph state $\ket{G}$ so that $G$ has a biclique decomposition of size $k$ can be constructed by a stabilizer circuit of size $O(k+n)$.
\end{theorem}
\begin{proof}
    We may assume that $G$ has at least two vertices. Let $u$ and $v$ be any two vertices of $G$, and let $\widetilde{G}$ be the graph obtained from $G$ by removing all of the edges incident to $u$ or~$v$. 
    
    We first show how to find a stabilizer circuit for $G$ given a stabilizer circuit for $\widetilde{G}$. For this step, we simply perform a $CZ$ gate on each pair of vertices $x$ and $y$ so that $xy$ is an edge of $G$ but not of $\widetilde{G}$. This step uses $O(n)$ gates.

    So it is enough to find a stabilizer circuit constructing $\widetilde{G}$. Note that there is a biclique decomposition $\mathcal{B}$ of $\widetilde{G}$ of size at most $k$. (We can obtain $\mathcal{B}$ by starting with a biclique decomposition of $G$ of size at most $k$ and then removing $u$ and $v$ from each of its bicliques.) Let $B$ be a biclique in $\mathcal{B}$, and suppose that we have a stabilizer circuit which constructs the graph $\widetilde{G}- E_B$ which is obtained from $\widetilde{G}$ by deleting all of the edges $E_B$ in $B$. We now show how to perform some additional $H$, $S$, and $CZ$ gates in order to obtain $\widetilde{G}$ from $\widetilde{G}- E_B$. A similar process can then be repeated to build $\widetilde{G}$ from the graph state with no edges, deleting one biclique at a time. (Note that the graph state with no edges can be constructed from the state $\ket{0}^{\otimes V}$ by performing $n$ Hadamard gates.)

    First, we perform $(|V(B)|+1)$-many $CZ$ gates to $\ket{\widetilde{G}-E_B}$ in order to make $u$ and $v$ adjacent, to make $u$ adjacent to one side of $B$, and to make $v$ adjacent to the other side of $B$. Thus, in this graph, the sum of the degrees of $u$ and $v$ equals $|V(B)|+2$. So by Lemma~\ref{lem:pivotStabilizer}, there is a stabilizer circuit of size $O(|V(B)|)$ which performs a pivot on the edge $uv$. Recall that pivoting switches the labels of $u$ and $v$, and switches adjacency between pairs of vertices which have an odd number of common neighbors in $u$ and $v$. Finally, we again perform $(|V(B)|+1)$-many $CZ$ gates in order to remove all of the edges incident to $u$ and $v$. In this manner we obtain a stabilizer circuit with $O(|V(B)|)$ many gates which constructs $\ket{\widetilde{G}}$ from $\ket{\widetilde{G}- E_B}$.

    The total size of the stabilizer circuit we have described is $O(k+n)$, as desired.
\end{proof}

\paragraph{Applications to classes of graphs.}

Now we show that interval graphs, circle graphs, and graphs of bounded rank-width have small biclique decompositions, thus recovering three of the results of~\cite{DJ25, KMY25} up to logarithmic factors. Graph classes of bounded rank-width also have bounded twin-width, so they have biclique decompositions of size $O(n \log^2{n})$ by Corollary~\ref{cor:width}. So it just remains to consider interval graphs and circle graphs.

Given a finite collection $\mathcal{I}$ of intervals on the real line, the corresponding \defn{interval graph} $G(\mathcal{I})$ has vertex set $\mathcal{I}$, and two vertices are adjacent if they intersect as intervals (i.e., if they have non-empty intersection). Similarly, the corresponding \defn{circle graph} (also called an \defn{overlap graph}) $\mathcal{C}(\mathcal{I})$ has vertex set $\mathcal{I}$, and two vertices are adjacent if they \defn{overlap} as intervals, meaning that they intersect but neither is contained in the other. We now prove that these graphs have reasonably small biclique decompositions, which will complete our quantum circuit complexity bounds. The proof relies on elementary results from~\cite{CY25}.

\begin{lemma}
\label{lem:intervalCircle}
    Every $n$-vertex interval graph has a biclique decomposition of size $O(n \log^2{n})$, and every $n$-vertex circle graph has a biclique decomposition of size $O(n \log^3{n})$.
\end{lemma}
\begin{proof}
    Let $\mathcal{I}$ be a collection of $n$ intervals, and consider its interval graph $G(\mathcal{I})$ and its circle graph $\mathcal{C}(\mathcal{I})$. We may assume that no two intervals in $\mathcal{I}$ have a common endpoint by changing the endpoints slightly. Consider a point $p \in \mathbb{R}$ so that half of the endpoints are to the left of $p$ and half are to the right. Let $\mathcal{L}$ and $\mathcal{R}$, respectively, be the set of all intervals in $\mathcal{I}$ which are contained to the left and right, respectively, of $p$. Thus at most half of the intervals in $\mathcal{I}$ are in each of $\mathcal{L}$ and $\mathcal{R}$.

    We now show how to find a biclique decomposition of the graph obtained from $G(\mathcal{I})$ and $\mathcal{C}(\mathcal{I})$, respectively, by making both $\mathcal{L}$ and $\mathcal{R}$ independent sets. These decompositions will have size $O(n\log{n})$ and $O(n\log^2{n})$, respectively. This will complete the proof as we can then continue inductively with the intervals $\mathcal{L}$ and $\mathcal{R}$.

    First we consider the intervals $\mathcal{P} \coloneqq \mathcal{I}\setminus (\mathcal{L} \cup \mathcal{R})$ which contain the point $p$. In the case of interval graphs, these intervals form a clique, which has a biclique decomposition of size $O(n\log{n})$, as desired. In the case of circle graphs, these intervals induce a \defn{permutation graph}.
    Permutation graphs are exactly the comparability graphs of partially ordered sets of dimension two, and from a result of Cardinal and Yuditsky~\cite[Theorem 9]{CY25}, they have biclique decompositions of size $O(n\log^2{n})$, as desired.

    We then need to partition the edges corresponding respectively to intersections and overlaps between an interval of $\mathcal{P}$ and an interval of $\mathcal{L}$ or $\mathcal{R}$.
    Without loss of generality, let us just consider the edges involving intervals of $\mathcal{L}$.
    In the case of interval graphs, since all intervals of $\mathcal{P}$ contain the point $p$, 
    they intersect an interval $\iota$ of $\mathcal{L}$ only if their left endpoint is on the left of the right endpoint of $\iota$. These edges therefore form a bipartite one-dimensional comparability graph, which from \cite[Theorem 8]{CY25} admits a biclique decomposition of size $O(n\log n)$. In the case of circle graphs, the bipartite overlap graph between intervals of $\mathcal{P}$ and of $\mathcal{L}$ can be seen to be a bipartite two-dimensional comparability graph. Indeed, an interval from $\mathcal{P}$ overlaps an interval $\iota$ from $\mathcal{L}$ only if its left endpoint is contained in $\iota$. 
    Again, from \cite[Theorem 8]{CY25}, these admit a biclique decomposition of size $O(n\log^2 n)$.

    The induction on the left and right subgraphs induced by $\mathcal{L}$ and $\mathcal{R}$ incur an additional logarithmic factor. We therefore obtain biclique decompositions of size $O(n \log^2{n})$ for interval graphs, and of size $O(n \log^3{n})$ for circle graphs, as claimed.
\end{proof}

\section*{Acknowledgments}
Much of this work was completed at the 13th Annual Workshop on Geometry and Graphs held at the Bellairs Research Institute in February 2026, and at the 52nd International Workshop on Graph-Theoretic Concepts in Computer Science held in Kortrijk in June 2026. We are grateful to the organizers and participants for providing an excellent research environment. 

\bibliographystyle{alpha}
\bibliography{bicliques}

@article{A24,
  author       = {Noga Alon},
  title        = {Implicit Representation of Sparse Hereditary Families},
  journal      = {Discret. Comput. Geom.},
  volume       = {72},
  number       = {2},
  pages        = {476--482},
  year         = {2024},
  url          = {https://doi.org/10.1007/s00454-023-00521-0},
  doi          = {10.1007/S00454-023-00521-0},
  bibsource    = {dblp computer science bibliography, https://dblp.org}
}

@article{AMY14,
  author       = {Noga Alon and
                  Shay Moran and
                  Amir Yehudayoff},
  title        = {Sign rank, {VC} dimension and spectral gaps},
  journal      = {Electron. Colloquium Comput. Complex.},
  volume       = {{TR14}},
  eid          = {{TR14-135}},
  year         = {2014},
  url          = {https://eccc.weizmann.ac.il/report/2014/135},
  eprinttype   = {ECCC},
  eprint       = {TR14-135},
  bibsource    = {dblp computer science bibliography, https://dblp.org}
}

@article{AES26,
	author = {Agarwal, Pankaj K. and Ezra, Esther and Sharir, Micha},
	doi = {10.1007/s00454-025-00792-9},
	isbn = {1432-0444},
	journal = {Discrete \& Computational Geometry},
	number = {2},
	pages = {320--342},
	title = {Semi-Algebraic Off-line Range Searching and Biclique Partitions in the Plane},
	url = {https://doi.org/10.1007/s00454-025-00792-9},
	volume = {75},
	year = {2026},}

@article {MR1417566,
    AUTHOR = {Fishburn, Peter C. and Hammer, Peter L.},
     TITLE = {Bipartite dimensions and bipartite degrees of graphs},
   JOURNAL = {Discrete Math.},
  FJOURNAL = {Discrete Mathematics},
    VOLUME = {160},
      YEAR = {1996},
    NUMBER = {1-3},
     PAGES = {127--148},
      ISSN = {0012-365X,1872-681X},
   MRCLASS = {05C70},
  MRNUMBER = {1417566},
MRREVIEWER = {R.\ Balakrishnan},
       DOI = {10.1016/0012-365X(95)00154-O},
       URL = {https://doi.org/10.1016/0012-365X(95)00154-O},
}

@inproceedings{DKP24,
  author       = {Lech Duraj and
                  Filip Konieczny and
                  Krzysztof Potepa},
  editor       = {Timothy M. Chan and
                  Johannes Fischer and
                  John Iacono and
                  Grzegorz Herman},
  title        = {Better Diameter Algorithms for Bounded {VC}-Dimension Graphs and Geometric
                  Intersection Graphs},
  booktitle    = {32nd Annual European Symposium on Algorithms, {ESA} 2024, Royal Holloway,
                  London, United Kingdom, September 2-4, 2024},
  series       = {LIPIcs},
  pages        = {51:1--51:18},
  publisher    = {Schloss Dagstuhl - Leibniz-Zentrum f{\"{u}}r Informatik},
  year         = {2024},
  url          = {https://doi.org/10.4230/LIPIcs.ESA.2024.51},
  doi          = {10.4230/LIPICS.ESA.2024.51},
  bibsource    = {dblp computer science bibliography, https://dblp.org}
}

@inproceedings{CCGKLZ25,
  author       = {Timothy M. Chan and
                  Hsien{-}Chih Chang and
                  Jie Gao and
                  S{\'{a}}ndor Kisfaludi{-}Bak and
                  Hung Le and
                  Da Wei Zheng},
  title        = {Truly Subquadratic Time Algorithms for Diameter and Related Problems
                  in Graphs of Bounded {VC}-dimension},
  booktitle    = {66th {IEEE} Annual Symposium on Foundations of Computer Science, {FOCS}
                  2025, Sydney, Australia, December 14-17, 2025},
  pages        = {2728--2765},
  publisher    = {{IEEE}},
  year         = {2025},
  url          = {https://doi.org/10.1109/FOCS63196.2025.00140},
  doi          = {10.1109/FOCS63196.2025.00140},
  bibsource    = {dblp computer science bibliography, https://dblp.org}
}

@article{KCA09,
  author       = {Chinmay Karande and
                  Kumar Chellapilla and
                  Reid Andersen},
  title        = {Speeding Up Algorithms on Compressed Web Graphs},
  journal      = {Internet Math.},
  volume       = {6},
  number       = {3},
  pages        = {373--398},
  year         = {2009},
  url          = {https://doi.org/10.1080/15427951.2009.10390646},
  doi          = {10.1080/15427951.2009.10390646},
  bibsource    = {dblp computer science bibliography, https://dblp.org}
}

@inproceedings{BC08,
  author       = {Gregory Buehrer and
                  Kumar Chellapilla},
  editor       = {Marc Najork and
                  Andrei Z. Broder and
                  Soumen Chakrabarti},
  title        = {A scalable pattern mining approach to web graph compression with communities},
  booktitle    = {Proceedings of the International Conference on Web Search and Web
                  Data Mining, {WSDM} 2008, Palo Alto, California, USA, February 11-12,
                  2008},
  pages        = {95--106},
  publisher    = {{ACM}},
  year         = {2008},
  url          = {https://doi.org/10.1145/1341531.1341547},
  doi          = {10.1145/1341531.1341547},
  bibsource    = {dblp computer science bibliography, https://dblp.org}
}

@article{DK26,
  author       = {Jan Dreier and
                  Clemens Kuske},
  title        = {Near-Linear Time Computation of Welzl Orders on Graphs with Linear
                  Neighborhood Complexity},
  journal      = {CoRR},
  volume       = {abs/2602.14625},
  year         = {2026},
  url          = {https://doi.org/10.48550/arXiv.2602.14625},
  doi          = {10.48550/ARXIV.2602.14625},
  eprinttype   = {arXiv},
  eprint       = {2602.14625},
  bibsource    = {dblp computer science bibliography, https://dblp.org}
}

@article{BGKM26,
  author       = {{\'{E}}douard Bonnet and
                  Colin Geniet and
                  Eun Jung Kim and
                  Sungmin Moon},
  title        = {Fast Shortest Path in Graphs With Sparse Signed Tree Models and Applications},
  journal      = {CoRR},
  volume       = {abs/2602.16605},
  year         = {2026},
  url          = {https://doi.org/10.48550/arXiv.2602.16605},
  doi          = {10.48550/ARXIV.2602.16605},
  eprinttype   = {arXiv},
  eprint       = {2602.16605},
  bibsource    = {dblp computer science bibliography, https://dblp.org},
  note = {To appear at ICALP'26}
}

@inproceedings{KLP26,
  author       = {Miroslaw Kowaluk and
                  Andrzej Lingas and
                  Mia Persson},
  editor       = {Christos D. Zaroliagis and
                  Dinabandhu Bhandari and
                  Prosenjit Gupta and
                  Swagatam Das},
  title        = {Approximate All-Pairs Hamming Distances and 0-1 Matrix Multiplication},
  booktitle    = {Applied Algorithms - Third International Conference, {ICAA} 2026,
                  Kolkata, India, January 7-9, 2026, Proceedings},
  series       = {Lecture Notes in Computer Science},
  pages        = {38--49},
  publisher    = {Springer},
  year         = {2026},
  url          = {https://doi.org/10.1007/978-3-032-15621-1\_4},
  doi          = {10.1007/978-3-032-15621-1\_4},
  bibsource    = {dblp computer science bibliography, https://dblp.org}
}

@inproceedings{CS26,
  author       = {Jean Cardinal and
                  Micha Sharir},
  title        = {Implicit representations via the polynomial method},
  year         = {2026},
  url          = {https://doi.org/10.48550/arXiv.2602.10922},
  booktitle = {Proceedings of the 52nd International Workshop on Graph-Theoretic Concepts in Computer Science},
  note = {To appear}
}

@article {MR65617,
    AUTHOR = {Tam\'{a}s K\"{o}vari and Vera T. S\'os and P\'{a}l Tur\'an},
     TITLE = {On a problem of {K}. {Z}arankiewicz},
   JOURNAL = {Colloq. Math.},
  FJOURNAL = {Colloquium Mathematicum},
    VOLUME = {3},
      YEAR = {1954},
     PAGES = {50--57},
      ISSN = {0010-1354,1730-6302},
   MRCLASS = {27.2X},
  MRNUMBER = {65617},
MRREVIEWER = {J.\ Riguet},
       DOI = {10.4064/cm-3-1-50-57},
       URL = {https://doi.org/10.4064/cm-3-1-50-57},
}

@article {MR1452952,
    AUTHOR = {Erd\H{o}s, Paul and Pyber, L\'{a}szl\'{o}},
     TITLE = {Covering a graph by complete bipartite graphs},
   JOURNAL = {Discrete Math.},
  FJOURNAL = {Discrete Mathematics},
    VOLUME = {170},
      YEAR = {1997},
    NUMBER = {1-3},
     PAGES = {249--251},
      ISSN = {0012-365X,1872-681X},
   MRCLASS = {05C70},
  MRNUMBER = {1452952},
       DOI = {10.1016/S0012-365X(96)00124-0},
       URL = {https://doi.org/10.1016/S0012-365X(96)00124-0},
}

@article{L56,
author = {Oleg Lupanov},
title = {On rectifier and switching-and-rectifier circuits},
journal = {Doklady Academii nauk SSSR},
year = {1956}, 
volume = {111}, 
number = {6}, 
pages = {1171-–1174}
}

@article{N69,
author = {Eduard Ivanovich Nechiporuk},
title = {The topological principles of self-correction},
journal = {Problemy Kibernet.},
year = {1969}, 
volume = {21}, 
pages = {5--102},
note = {(In Russian)}
}

@article{DEGM05,
  author       = {Matthew Dickerson and
                  David Eppstein and
                  Michael T. Goodrich and
                  Jeremy Yu Meng},
  title        = {Confluent Drawings: Visualizing Non-planar Diagrams in a Planar Way},
  journal      = {J. Graph Algorithms Appl.},
  volume       = {9},
  number       = {1},
  pages        = {31--52},
  year         = {2005},
  url          = {https://doi.org/10.7155/jgaa.00099},
  doi          = {10.7155/JGAA.00099},
  bibsource    = {dblp computer science bibliography, https://dblp.org}
}

@article{BGKTW22,
  author       = {{\'{E}}douard Bonnet and
                  Colin Geniet and
                  Eun Jung Kim and
                  St{\'{e}}phan Thomass{\'{e}} and
                  R{\'{e}}mi Watrigant},
  title        = {Twin-width {II:} small classes},
  journal      = {Comb. Theory},
  volume       = {2},
  number       = {2},
  year         = {2022},
  url          = {https://doi.org/10.5070/c62257876},
  doi          = {10.5070/C62257876},
  bibsource    = {dblp computer science bibliography, https://dblp.org}
}

@article{BGKTW24,
  author       = {{\'{E}}douard Bonnet and
                  Colin Geniet and
                  Eun Jung Kim and
                  St{\'{e}}phan Thomass{\'{e}} and
                  R{\'{e}}mi Watrigant},
  title        = {Twin-Width {III:} Max Independent Set, Min Dominating Set, and Coloring},
  journal      = {{SIAM} J. Comput.},
  volume       = {53},
  number       = {5},
  pages        = {1602--1640},
  year         = {2024},
  url          = {https://doi.org/10.1137/21m142188x},
  doi          = {10.1137/21M142188X},
  bibsource    = {dblp computer science bibliography, https://dblp.org}
}

@article{BGMSTT24,
  author       = {{\'{E}}douard Bonnet and
                  Ugo Giocanti and
                  Patrice Ossona de Mendez and
                  Pierre Simon and
                  St{\'{e}}phan Thomass{\'{e}} and
                  Szymon Torunczyk},
  title        = {Twin-Width {IV:} Ordered Graphs and Matrices},
  journal      = {J. {ACM}},
  volume       = {71},
  number       = {3},
  pages        = {21},
  year         = {2024},
  url          = {https://doi.org/10.1145/3651151},
  doi          = {10.1145/3651151},
  bibsource    = {dblp computer science bibliography, https://dblp.org}
}

@inproceedings{T23,
  author       = {Szymon Toru\'{n}czyk},
  title        = {Flip-width: Cops and Robber on dense graphs},
  booktitle    = {64th {IEEE} Annual Symposium on Foundations of Computer Science, {FOCS}
                  2023, Santa Cruz, CA, USA, November 6-9, 2023},
  pages        = {663--700},
  publisher    = {{IEEE}},
  year         = {2023},
  url          = {https://doi.org/10.1109/FOCS57990.2023.00045},
  doi          = {10.1109/FOCS57990.2023.00045},
  bibsource    = {dblp computer science bibliography, https://dblp.org}
}

@inproceedings{DT25,
  author       = {Jan Dreier and
                  Szymon Toru\'{n}czyk},
  editor       = {Michal Kouck{\'{y}} and
                  Nikhil Bansal},
  title        = {Merge-Width and First-Order Model Checking},
  booktitle    = {Proceedings of the 57th Annual {ACM} Symposium on Theory of Computing,
                  {STOC} 2025, Prague, Czechia, June 23-27, 2025},
  pages        = {1944--1955},
  publisher    = {{ACM}},
  year         = {2025},
  url          = {https://doi.org/10.1145/3717823.3718259},
  doi          = {10.1145/3717823.3718259},
  bibsource    = {dblp computer science bibliography, https://dblp.org}
}

@article{KPSS25,
  author       = {Andrew Krapivin and
                  Benjamin Przybocki and
                  Nicol{\'{a}}s Sanhueza{-}Matamala and
                  Bernardo Subercaseaux},
  title        = {Optimal and Efficient Partite Decompositions of Hypergraphs},
  journal      = {CoRR},
  volume       = {abs/2511.11855},
  year         = {2025},
  url          = {https://doi.org/10.48550/arXiv.2511.11855},
  doi          = {10.48550/ARXIV.2511.11855},
  eprinttype    = {arXiv},
  eprint       = {2511.11855},
  bibsource    = {dblp computer science bibliography, https://dblp.org},
  note = {To appear in {STOC}'26}
}

@inproceedings{
    AvdBM25,
    title={The Structural Complexity of Matrix-Vector Multiplication},
    author={Emile Timothy Anand and Jan van den Brand and Rose McCarty},
    booktitle={The Thirty-ninth Annual Conference on Neural Information Processing Systems},
    year={2025},
    url={https://openreview.net/forum?id=tGLZj8GWx3}
}

@incollection {MR4868423,
    AUTHOR = {Bonnet, \'Edouard and Duron, Julien and Sylvester, John and
              Zamaraev, Viktor},
     TITLE = {Adjacency labeling schemes for small classes},
 BOOKTITLE = {16th {I}nnovations in {T}heoretical {C}omputer {S}cience
              {C}onference},
    SERIES = {LIPIcs. Leibniz Int. Proc. Inform.},
    VOLUME = {325},
     PAGES = {Art. No. 21, 22},
 PUBLISHER = {Schloss Dagstuhl. Leibniz-Zent. Inform., Wadern},
      YEAR = {2025},
      ISBN = {978-3-95977-361-4},
   MRCLASS = {68R10},
  MRNUMBER = {4868423},
       DOI = {10.4230/lipics.itcs.2025.21},
       URL = {https://doi.org/10.4230/lipics.itcs.2025.21},
}

@incollection {MR4802148,
    AUTHOR = {Bonnet, \'Edouard and Duron, Julien and Sylvester, John and
              Zamaraev, Viktor},
     TITLE = {Symmetric-difference (degeneracy) and signed tree models},
 BOOKTITLE = {49th {I}nternational {S}ymposium on {M}athematical
              {F}oundations of {C}omputer {S}cience},
    SERIES = {LIPIcs. Leibniz Int. Proc. Inform.},
    VOLUME = {306},
     PAGES = {Art. No. 32, 16},
 PUBLISHER = {Schloss Dagstuhl. Leibniz-Zent. Inform., Wadern},
      YEAR = {2024},
      ISBN = {978-3-95977-335-5},
   MRCLASS = {68R10},
  MRNUMBER = {4802148},
       DOI = {10.4230/lipics.mfcs.2024.32},
       URL = {https://doi.org/10.4230/lipics.mfcs.2024.32},
}

@article {MR4623907,
    AUTHOR = {Bonnet, \'Edouard and Foucaud, Florent and Lehtil\"a, Tuomo
              and Parreau, Aline},
     TITLE = {Neighbourhood complexity of graphs of bounded twin-width},
   JOURNAL = {European J. Combin.},
  FJOURNAL = {European Journal of Combinatorics},
    VOLUME = {115},
      YEAR = {2024},
     PAGES = {Paper No. 103772, 8},
      ISSN = {0195-6698,1095-9971},
   MRCLASS = {05C85},
  MRNUMBER = {4623907},
MRREVIEWER = {Guantao\ Chen},
       DOI = {10.1016/j.ejc.2023.103772},
       URL = {https://doi.org/10.1016/j.ejc.2023.103772},
}

@article{BG25,
  author       = {Marthe Bonamy and
                  Colin Geniet},
  title        = {{\(\chi\)}-Boundedness and Neighbourhood Complexity of Bounded Merge-Width
                  Graphs},
  journal      = {CoRR},
  volume       = {abs/2504.08266},
  year         = {2025},
  url          = {https://doi.org/10.48550/arXiv.2504.08266},
  doi          = {10.48550/ARXIV.2504.08266},
  eprinttype    = {arXiv},
  eprint       = {2504.08266},
  bibsource    = {dblp computer science bibliography, https://dblp.org}
}

@inproceedings {MR1213461,
    AUTHOR = {Welzl, Emo},
     TITLE = {Partition trees for triangle counting and other range
              searching problems},
 BOOKTITLE = {Proceedings of the {F}ourth {A}nnual {S}ymposium on
              {C}omputational {G}eometry ({U}rbana, {IL}, 1988)},
     PAGES = {23--33},
 PUBLISHER = {ACM, New York},
      YEAR = {1988},
      ISBN = {0-89791-270-5},
   MRCLASS = {68U05 (68P10)},
  MRNUMBER = {1213461},
       DOI = {10.1145/73393.73397},
       URL = {https://doi.org/10.1145/73393.73397},
}

@article {MR1014739,
    AUTHOR = {Chazelle, Bernard and Welzl, Emo},
     TITLE = {Quasi-optimal range searching in spaces of finite
              {VC}-dimension},
   JOURNAL = {Discrete Comput. Geom.},
  FJOURNAL = {Discrete \& Computational Geometry. An International Journal
              of Mathematics and Computer Science},
    VOLUME = {4},
      YEAR = {1989},
    NUMBER = {5},
     PAGES = {467--489},
      ISSN = {0179-5376,1432-0444},
   MRCLASS = {68U05},
  MRNUMBER = {1014739},
       DOI = {10.1007/BF02187743},
       URL = {https://doi.org/10.1007/BF02187743},
}

@article {MR1313896,
    AUTHOR = {Haussler, David},
     TITLE = {Sphere packing numbers for subsets of the {B}oolean {$n$}-cube
              with bounded {V}apnik-{C}hervonenkis dimension},
   JOURNAL = {J. Combin. Theory Ser. A},
  FJOURNAL = {Journal of Combinatorial Theory. Series A},
    VOLUME = {69},
      YEAR = {1995},
    NUMBER = {2},
     PAGES = {217--232},
      ISSN = {0097-3165,1096-0899},
   MRCLASS = {52C17 (60F17)},
  MRNUMBER = {1313896},
MRREVIEWER = {Evarist\ Gin\'e},
       DOI = {10.1016/0097-3165(95)90052-7},
       URL = {https://doi.org/10.1016/0097-3165(95)90052-7},
}

@article {MR1262921,
    AUTHOR = {Matou\v{s}ek, Ji\v{r}\'{i} and Welzl, Emo and Wernisch, Lorenz},
     TITLE = {Discrepancy and approximations for bounded {VC}-dimension},
   JOURNAL = {Combinatorica},
  FJOURNAL = {Combinatorica. An International Journal on Combinatorics and
              the Theory of Computing},
    VOLUME = {13},
      YEAR = {1993},
    NUMBER = {4},
     PAGES = {455--466},
      ISSN = {0209-9683},
   MRCLASS = {52A37 (05A05)},
  MRNUMBER = {1262921},
MRREVIEWER = {B\'ela\ Uhrin},
       DOI = {10.1007/BF01303517},
       URL = {https://doi.org/10.1007/BF01303517},
}

@book {MR2683232,
    AUTHOR = {Matou\v{s}ek, Ji\v{r}\'{i}},
     TITLE = {Geometric discrepancy},
    SERIES = {Algorithms and Combinatorics},
    VOLUME = {18},
      NOTE = {An illustrated guide,
              Revised paperback reprint of the 1999 original},
 PUBLISHER = {Springer-Verlag, Berlin},
      YEAR = {2010},
     PAGES = {xiv+296},
      ISBN = {978-3-642-03941-6},
   MRCLASS = {11K38 (05D05 52C99 65D18)},
  MRNUMBER = {2683232},
       DOI = {10.1007/978-3-642-03942-3},
       URL = {https://doi.org/10.1007/978-3-642-03942-3},
}

@article {MR3646875,
    AUTHOR = {Fox, Jacob and Pach, J\'anos and Sheffer, Adam and Suk, Andrew
              and Zahl, Joshua},
     TITLE = {A semi-algebraic version of {Z}arankiewicz's problem},
   JOURNAL = {J. Eur. Math. Soc. (JEMS)},
  FJOURNAL = {Journal of the European Mathematical Society (JEMS)},
    VOLUME = {19},
      YEAR = {2017},
    NUMBER = {6},
     PAGES = {1785--1810},
      ISSN = {1435-9855,1435-9863},
   MRCLASS = {05C35 (14P10 52C10 68R10)},
  MRNUMBER = {3646875},
MRREVIEWER = {J.\ C.\ Lagarias},
       DOI = {10.4171/JEMS/705},
}

@inproceedings{CCCK24,
  author       = {Sergio Cabello and
                  Siu{-}Wing Cheng and
                  Otfried Cheong and
                  Christian Knauer},
  title        = {Geometric Matching and Bottleneck Problems},
  booktitle    = {40th International Symposium on Computational Geometry, SoCG 2024,
                  June 11-14, 2024, Athens, Greece},
  pages        = {31:1--31:15},
  year         = {2024},
  doi          = {10.4230/LIPICS.SOCG.2024.31},
}

@incollection {MR820214,
    AUTHOR = {Chung, Fan R. K. and Erd\H{o}s, Paul and Spencer, Joel},
     TITLE = {On the decomposition of graphs into complete bipartite
              subgraphs},
 BOOKTITLE = {Studies in pure mathematics},
     PAGES = {95--101},
 PUBLISHER = {Birkh\"auser, Basel},
      YEAR = {1983},
      ISBN = {3-7643-1288-2},
   MRCLASS = {05C35 (05C80)},
  MRNUMBER = {820214},
MRREVIEWER = {Micha\l\ Karo\'nski},
doi = {https://doi.org/10.1007/978-3-0348-5438-2_10},
}

@article {MR1739610,
    AUTHOR = {Agarwal, Pankaj K. and Varadarajan, Kasturi R.},
     TITLE = {Efficient algorithms for approximating polygonal chains},
   JOURNAL = {Discrete Comput. Geom.},
  FJOURNAL = {Discrete \& Computational Geometry. An International Journal
              of Mathematics and Computer Science},
    VOLUME = {23},
      YEAR = {2000},
    NUMBER = {2},
     PAGES = {273--291},
      ISSN = {0179-5376,1432-0444},
   MRCLASS = {52B55 (52A27 68U05 68W20)},
  MRNUMBER = {1739610},
       DOI = {10.1007/PL00009500},
}

@article {MR1471987,
    AUTHOR = {Katz, Matthew J. and Sharir, Micha},
     TITLE = {An expander-based approach to geometric optimization},
   JOURNAL = {SIAM J. Comput.},
  FJOURNAL = {SIAM Journal on Computing},
    VOLUME = {26},
      YEAR = {1997},
    NUMBER = {5},
     PAGES = {1384--1408},
      ISSN = {0097-5397},
   MRCLASS = {90C28 (52B55)},
  MRNUMBER = {1471987},
MRREVIEWER = {J.\ S\'andor},
       DOI = {10.1137/S0097539794268649},
}

@article {MR4013919,
    AUTHOR = {Do, Thao},
     TITLE = {Representation complexities of semialgebraic graphs},
   JOURNAL = {SIAM J. Discrete Math.},
  FJOURNAL = {SIAM Journal on Discrete Mathematics},
    VOLUME = {33},
      YEAR = {2019},
    NUMBER = {4},
     PAGES = {1864--1877},
      ISSN = {0895-4801,1095-7146},
   MRCLASS = {05C62 (05C65 05C70 14P10 52C10)},
  MRNUMBER = {4013919},
MRREVIEWER = {Jen\H o\ Lehel},
       DOI = {10.1137/18M1221606},
}

@article{AAEKS25,
  author       = {Pankaj K. Agarwal and
                  Boris Aronov and
                  Esther Ezra and
                  Matthew J. Katz and
                  Micha Sharir},
  title        = {Intersection Queries for Flat Semi-Algebraic Objects in Three Dimensions
                  and Related Problems},
  journal      = {{ACM} Trans. Algorithms},
  volume       = {21},
  number       = {3},
  pages        = {25:1--25:59},
  year         = {2025},
  url          = {https://doi.org/10.1145/3721290},
  doi          = {10.1145/3721290},
  bibsource    = {dblp computer science bibliography, https://dblp.org}
}

@article {FM95,
    AUTHOR = {Feder, Tom\'as and Motwani, Rajeev},
     TITLE = {Clique partitions, graph compression and speeding-up
              algorithms},
   JOURNAL = {J. Comput. System Sci.},
  FJOURNAL = {Journal of Computer and System Sciences},
    VOLUME = {51},
      YEAR = {1995},
    NUMBER = {2},
     PAGES = {261--272},
      ISSN = {0022-0000,1090-2724},
   MRCLASS = {68R10 (68Q20)},
  MRNUMBER = {1356505},
       DOI = {10.1006/jcss.1995.1065},
}

@article{CT22,
  author       = {Jonathan B. Conroy and
                  Csaba D. T{\'{o}}th},
  title        = {Hop-spanners for geometric intersection graphs},
  journal      = {J. Comput. Geom.},
  volume       = {14},
  number       = {2},
  pages        = {26--64},
  year         = {2022},
  doi          = {10.20382/JOCG.V14I2A3},
  bibsource    = {dblp computer science bibliography, https://dblp.org}
}

@article{MR1298916,
    AUTHOR = {Agarwal, Pankaj K. and Alon, Noga and Aronov, Boris and Suri, Subash},
     TITLE = {Can visibility graphs be represented compactly?},
   JOURNAL = {Discrete Comput. Geom.},
  FJOURNAL = {Discrete \& Computational Geometry. An International Journal
              of Mathematics and Computer Science},
    VOLUME = {12},
      YEAR = {1994},
    NUMBER = {3},
     PAGES = {347--365},
      ISSN = {0179-5376,1432-0444},
   MRCLASS = {68U05 (05C75)},
  MRNUMBER = {1298916},
MRREVIEWER = {Igor\ Rivin},
       DOI = {10.1007/BF02574385},
}

@inproceedings{M92,
  title={Range searching with efficient hierarchical cuttings},
  author={Matou{\v{s}}ek, Ji{\v{r}}{\'\i}},
  booktitle={Proceedings of the eighth annual symposium on Computational geometry},
  pages={276--285},
  year={1992},
 doi          = {10.1145/142675.142732},
}

@article{E96,
  title={New lower bounds for {H}opcroft's problem},
  author={Erickson, Jeff},
  journal={Discrete \& Computational Geometry},
  volume={16},
  number={4},
  pages={389--418},
  year={1996},
  publisher={Springer-Verlag Berlin, Heidelberg},
doi          = {10.1007/BF02712875},
}

@article{BK03,
  title={On counting point-hyperplane incidences},
  author={Bra{\ss}, Peter and Knauer, Christian},
  journal={Computational Geometry},
  volume={25},
  number={1-2},
  pages={13--20},
  year={2003},
  publisher={Elsevier},
  doi          = {10.1016/S0925-7721(02)00127-X},
}

@article{cha06,
  title={Dynamic subgraph connectivity with geometric applications},
  author={Chan, Timothy M.},
  journal={SIAM Journal on Computing},
  volume={36},
  number={3},
  pages={681--694},
  year={2006},
  publisher={SIAM},
doi          = {10.1137/S009753970343912X},
}

@article{BKTW21,
  title={Twin-width {I}: tractable FO model checking},
  author={Bonnet, {\'E}douard and Kim, Eun Jung and Thomass{\'e}, St{\'e}phan and Watrigant, R{\'e}mi},
  journal={ACM Journal of the ACM (JACM)},
  volume={69},
  number={1},
  pages={1--46},
  year={2021},
  publisher={ACM New York, NY},
  doi          = {10.1145/3486655},
}

@article {MR1417348,
    AUTHOR = {Koll\'ar, J\'anos and R\'onyai, Lajos and Szab\'o, Tibor},
     TITLE = {Norm-graphs and bipartite {T}ur\'an numbers},
   JOURNAL = {Combinatorica},
  FJOURNAL = {Combinatorica. An International Journal on Combinatorics and
              the Theory of Computing},
    VOLUME = {16},
      YEAR = {1996},
    NUMBER = {3},
     PAGES = {399--406},
      ISSN = {0209-9683},
   MRCLASS = {05C35},
  MRNUMBER = {1417348},
MRREVIEWER = {W.\ G.\ Brown},
       DOI = {10.1007/BF01261323},
       URL = {https://doi.org/10.1007/BF01261323},
}

@article{S24,
  title={A survey of {Z}arankiewicz problems in geometry},
  author={Smorodinsky, Shakhar},
  journal={CoRR},
  volume = {abs/2410.03702},
  year={2024},
  doi = {https://doi.org/10.48550/arXiv.2410.03702},
}

@inproceedings{CY25,
  author       = {Jean Cardinal and
                  Yelena Yuditsky},
  editor       = {Anne Benoit and
                  Haim Kaplan and
                  Sebastian Wild and
                  Grzegorz Herman},
  title        = {Compact Representation of Semilinear and Terrain-Like Graphs},
  booktitle    = {33rd Annual European Symposium on Algorithms, {ESA} 2025, Warsaw,
                  Poland, September 15-17, 2025},
  series       = {LIPIcs},
  pages        = {67:1--67:19},
  publisher    = {Schloss Dagstuhl - Leibniz-Zentrum f{\"{u}}r Informatik},
  year         = {2025},
  url          = {https://doi.org/10.4230/LIPIcs.ESA.2025.67},
  doi          = {10.4230/LIPICS.ESA.2025.67},
  bibsource    = {dblp computer science bibliography, https://dblp.org}
}

@InProceedings{BGOdMT23,
  author =	{Bonnet, \'{E}douard and Giocanti, Ugo and Ossona de Mendez, Patrice and Thomass\'{e}, St\'{e}phan},
  title =	{{Twin-Width V: Linear Minors, Modular Counting, and Matrix Multiplication}},
  booktitle =	{40th International Symposium on Theoretical Aspects of Computer Science (STACS 2023)},
  pages =	{15:1--15:16},
  series =	{Leibniz International Proceedings in Informatics (LIPIcs)},
  ISBN =	{978-3-95977-266-2},
  ISSN =	{1868-8969},
  year =	{2023},
  volume =	{254},
  editor =	{Berenbrink, Petra and Bouyer, Patricia and Dawar, Anuj and Kant\'{e}, Mamadou Moustapha},
  publisher =	{Schloss Dagstuhl -- Leibniz-Zentrum f{\"u}r Informatik},
  address =	{Dagstuhl, Germany},
  URL =		{https://drops.dagstuhl.de/entities/document/10.4230/LIPIcs.STACS.2023.15},
  URN =		{urn:nbn:de:0030-drops-176675},
  doi =		{10.4230/LIPIcs.STACS.2023.15}
}

@article{KO26,
      title={Fast and simple multiplication of bounded twin-width matrices}, 
      author={László Kozma and Michal Opler},
      year={2026},
      eprint={2602.20023},
      archivePrefix={arXiv},
      primaryClass={cs.DS},
      url={https://arxiv.org/abs/2602.20023}, 
  journal      = {CoRR},
  volume       = {abs/2602.20023},
}

@article {Sauer1972,
    AUTHOR = {Sauer, Norbert W.},
     TITLE = {On the density of families of sets},
   JOURNAL = {J. Combinatorial Theory Ser. A},
  FJOURNAL = {Journal of Combinatorial Theory. Series A},
    VOLUME = {13},
      YEAR = {1972},
     PAGES = {145--147},
      ISSN = {0097-3165},
       DOI = {10.1016/0097-3165(72)90019-2},
       URL = {https://doi.org/10.1016/0097-3165(72)90019-2},
}

@article {Shelah1972,
    AUTHOR = {Shelah, Saharon},
     TITLE = {A combinatorial problem; stability and order for models and
              theories in infinitary languages},
   JOURNAL = {Pacific J. Math.},
  FJOURNAL = {Pacific Journal of Mathematics},
    VOLUME = {41},
      YEAR = {1972},
     PAGES = {247--261},
      ISSN = {0030-8730,1945-5844},
       URL = {http://projecteuclid.org/euclid.pjm/1102968432},
}

@article{OneWayQC,
  title = {A One-Way Quantum Computer},
  author = {Raussendorf, Robert and Briegel, Hans J.},
  journal = {Phys. Rev. Lett.},
  volume = {86},
  issue = {22},
  pages = {5188--5191},
  numpages = {0},
  year = {2001},
  month = {May},
  publisher = {American Physical Society},
  doi = {10.1103/PhysRevLett.86.5188},
  url = {https://link.aps.org/doi/10.1103/PhysRevLett.86.5188}
}

@article{transformingStates,
	author = {Dahlberg, A.  and Wehner, S.},
	title = {Transforming graph states using single-qubit operations},
	journal = {Philosophical Transactions of the Royal Society A: Mathematical, Physical and Engineering Sciences},
	volume = {376},
	number = {2123},
	pages = {20170325},
	year = {2018},
	doi = {10.1098/rsta.2017.0325},
	URL = {https://royalsocietypublishing.org/doi/abs/10.1098/rsta.2017.0325},
	eprint = {https://royalsocietypublishing.org/doi/pdf/10.1098/rsta.2017.0325}
}

@article{graphStatesSecretSharing,
  title = {Graph states for quantum secret sharing},
  author = {Markham, Damian and Sanders, Barry C.},
  journal = {Phys. Rev. A},
  volume = {78},
  issue = {4},
  pages = {042309},
  numpages = {17},
  year = {2008},
  month = {Oct},
  publisher = {American Physical Society},
  doi = {10.1103/PhysRevA.78.042309},
  url = {https://link.aps.org/doi/10.1103/PhysRevA.78.042309}
}

@article{quantumNetworkRoutingAndLC,
    author = {Frederik Hahn  and Anna Pappa and Jens Eisert},
    title = {Quantum network routing and local complementation},
    journal = {npj Quantum Information},
    volume = {5},
    number = {76},
    year = {2019},
    doi = {10.1038/s41534-019-0191-6},
    URL = {https://doi.org/10.1038/s41534-019-0191-6}
}

@article{CliffordErrorCorrection,
  title = {Theory of fault-tolerant quantum computation},
  author = {Gottesman, Daniel},
  journal = {Phys. Rev. A},
  volume = {57},
  issue = {1},
  pages = {127--137},
  numpages = {0},
  year = {1998},
  month = {Jan},
  publisher = {American Physical Society},
  doi = {10.1103/PhysRevA.57.127},
  url = {https://link.aps.org/doi/10.1103/PhysRevA.57.127}
}

@article{optimalSynthesis,
    author = {Patel, Ketan N. and Markov, Igor L. and Hayes, John P.},
    title = {Optimal synthesis of linear reversible circuits},
    year = {2008},
    issue_date = {March 2008},
    publisher = {Rinton Press, Incorporated},
    address = {Paramus, NJ},
    volume = {8},
    number = {3},
    issn = {1533-7146},
    journal = {Quantum Info. Comput.},
    month = mar,
    pages = {282–294},
    numpages = {13}
}

@ARTICLE{circuitLowerBound,
  author={Shende, Vivek V. and Prasad, Aditya K. and Markov, Igor L. and Hayes, John P.},
  journal={IEEE Transactions on Computer-Aided Design of Integrated Circuits and Systems}, 
  title={Synthesis of reversible logic circuits}, 
  year={2003},
  volume={22},
  number={6},
  pages={710-722},
  doi={10.1109/TCAD.2003.811448}
}

@article{DJ25,
  author       = {James Davies and Andrew Jena},
  title        = {Preparing graph states forbidding a vertex-minor},
  journal      = {CoRR},
  volume       = {abs/2504.00291v2},
  year         = {2025},
  url          = {https://doi.org/10.48550/arXiv.2504.00291},
  doi          = {10.48550/arXiv.2504.00291},
  eprinttype   = {arXiv},
  eprint       = {2504.00291v2},
}

@article{KMY25,
  author       = {Soh Kumabe and Ryuhei Mori and Yusei Yoshimura},
  title        = {Complexity of graph-state preparation by {C}lifford circuits},
  journal      = {CoRR},
  volume       = {abs/2402.05874
v3},
  year         = {2025},
  url          = {https://doi.org/10.48550/arXiv.2402.05874},
  doi          = {10.48550/arXiv.2402.05874},
  eprinttype   = {arXiv},
  eprint       = {2402.05874v3},
}

@incollection {graphStateEntanglement,
    AUTHOR = {Marc Hein and Wolfgang Dur and Jens Eisert and Robert Raussendorf and Maarten Van den Nest and Hans J. Briegel.},
     TITLE = {Entanglement in Graph States and its Applications},
 BOOKTITLE = {Quantum Computers, Algorithms and Chaos},
    SERIES = {ProQuest Ebook Central},
    VOLUME = {162},
     PAGES = {115--218},
 PUBLISHER = {Sage Publications Ltd.},
      YEAR = {2006},
      ISBN = {9781614990185}
}

@article{OtherDistributing,
  title = {Distributing graph states over arbitrary quantum networks},
  author = {Meignant, Cl\'ement and Markham, Damian and Grosshans, Fr\'ed\'eric},
  journal = {Phys. Rev. A},
  volume = {100},
  issue = {5},
  pages = {052333},
  year = {2019},
  month = {Nov},
  publisher = {American Physical Society},
  doi = {10.1103/PhysRevA.100.052333},
  url = {https://link.aps.org/doi/10.1103/PhysRevA.100.052333}
}

@article{improvedSimulation04,
  title = {Improved simulation of stabilizer circuits},
  author = {Aaronson, Scott and Gottesman, Daniel},
  journal = {Phys. Rev. A},
  volume = {70},
  issue = {5},
  pages = {052328},
  numpages = {14},
  year = {2004},
  month = {Nov},
  publisher = {American Physical Society},
  doi = {10.1103/PhysRevA.70.052328},
  url = {https://link.aps.org/doi/10.1103/PhysRevA.70.052328}
}

@article{fastSimulationGS06,
  title = {Fast simulation of stabilizer circuits using a graph-state representation},
  author = {Anders, Simon and Briegel, Hans J.},
  journal = {Phys. Rev. A},
  volume = {73},
  issue = {2},
  pages = {022334},
  numpages = {9},
  year = {2006},
  month = {Feb},
  publisher = {American Physical Society},
  doi = {10.1103/PhysRevA.73.022334},
  url = {https://link.aps.org/doi/10.1103/PhysRevA.73.022334}
}

@article{Gottesman98,
  author       = {Daniel Gottesman},
  title        = {The Heisenberg Representation of Quantum Computers},
  journal      = {CoRR},
  volume       = {abs/quant-ph/9807006},
  year         = {1998},
  url          = {https://doi.org/10.48550/arXiv.quant-ph/9807006},
  doi          = {10.48550/arXiv.quant-ph/9807006},
  eprinttype   = {arXiv},
  eprint       = {9807006},
}

@article{LCEquiv,
  title = {Graphical description of the action of local {C}lifford transformations on graph states},
  author = {Van den Nest{}, Maarten and Dehaene, Jeroen and De Moor, Bart},
  journal = {Phys. Rev. A},
  volume = {69},
  issue = {2},
  pages = {022316},
  numpages = {7},
  year = {2004},
  month = {Feb},
  publisher = {American Physical Society},
  doi = {10.1103/PhysRevA.69.022316},
  url = {https://link.aps.org/doi/10.1103/PhysRevA.69.022316}
}

@article {RWAndVM,
    AUTHOR = {Oum, Sang{-}il},
     TITLE = {Rank-width and vertex-minors},
   JOURNAL = {J. Combin. Theory Ser. B},
  FJOURNAL = {Journal of Combinatorial Theory. Series B},
    VOLUME = {95},
      YEAR = {2005},
    NUMBER = {1},
     PAGES = {79--100},
      ISSN = {0095-8956},
   MRCLASS = {05C83 (05B35)},
  MRNUMBER = {2156341},
MRREVIEWER = {Maruti M. Shikare},
       DOI = {10.1016/j.jctb.2005.03.003},
}

@article{CLXLLH25,
  author       = {Kai Chen and Wen Liu and GuoSheng Xu and Yangzhi Li and Maoduo Li and Shouli He},
  title        = {Quantum Circuit Optimization Based on Dynamic Grouping and {ZX}-Calculus for Reducing 2-Qubit Gate Count},
  journal      = {CoRR},
  volume       = {abs/2507.14434},
  year         = {2025},
  url          = {https://doi.org/10.48550/arXiv.2507.14434},
  doi          = {10.48550/arXiv.2507.14434},
  eprinttype   = {arXiv},
  eprint       = {2507.14434},
}

@article{JWFL26,
  author       = {Tingxiang Ji and Hansika Weerasena and Demitry Farfurnik and Jianqing Liu},
  title        = {Towards Efficient Synthesis of Quantum Graph States by Fusing Graph Motifs},
  journal      = {CoRR},
  volume       = {abs/2606.02880},
  year         = {2026},
  url          = {https://doi.org/10.48550/arXiv.2606.02880},
  doi          = {10.48550/arXiv.2606.02880},
  eprinttype   = {arXiv},
  eprint       = {2606.02880},
}

@article{NGClifford2024,
  title = {Twisty-puzzle-inspired approach to {C}lifford synthesis},
  author = {Bao, Ning and Hartnett, Gavin S.},
  journal = {Phys. Rev. A},
  volume = {109},
  issue = {3},
  pages = {032409},
  numpages = {10},
  year = {2024},
  month = {Mar},
  publisher = {American Physical Society},
  doi = {10.1103/PhysRevA.109.032409},
  url = {https://link.aps.org/doi/10.1103/PhysRevA.109.032409}
}

@article{minimizingLCEquiv,
  author = {Hemant Sharma and Kenneth Goodenough and Johannes Borregaard and Filip Rozpędek and Jonas Helsen},
  title = {{M}inimising the number of edges in {LC}-equivalent graph states},
  year = {2026},
  eprint = {2506.00292},
  journal = {Quantum},
  url = {https://doi.org/10.22331/q-2026-02-09-2001},
  volume = {10}
}

@article{alon1999norm,
  title={Norm-graphs: variations and applications},
  author={Alon, Noga and R{\'o}nyai, Lajos and Szab{\'o}, Tibor},
  journal={Journal of Combinatorial Theory, Series B},
  volume={76},
  number={2},
  pages={280--290},
  year={1999},
  publisher={Elsevier}
}

@book{Shelah1982Book,
    author = {Shelah{}, Saharon},
    address = {Amsterdam},
    booktitle = {Classification theory and the number of non-isomorphic models},
    isbn = {0720407575},
    language = {eng},
    lccn = {77015646},
    year = {1978},
    publisher = {North-Holland Pub. Co.},
    series = {Studies in logic and the foundations of mathematics v. 92},
    title = {Classification theory and the number of non-isomorphic models},
}

@article{monDep26,
  author       = {Jan Dreier and Nikolas M{\"a}hlmann and Rose McCarty and Michał Pilipczuk and Szymon Toruńczyk},
  title        = {Neighborhood Complexity and Radius-1 Merge-Width in Monadically Dependent Graph Classes},
  journal      = {in preparation},
  year         = {2026}
}

\end{document}